\documentclass{amsart}
\usepackage{amssymb,stix,bbm}
\usepackage[matrix,arrow,curve]{xy}
\usepackage{verbatim}
\usepackage{graphics}
\usepackage{a4wide,array}
\usepackage[utf8]{inputenc}
\usepackage{tikz}
\usepackage{tikz-cd} 

\newtheorem{theorem}{Theorem}[section]
\newtheorem{lemma}[theorem]{Lemma}
\newtheorem{remark}[theorem]{Remark}

\newtheorem{conjecture}[theorem]{Conjecture}

\newtheorem{proposition}[theorem]{Proposition}

\usepackage{hyperref}
\hypersetup{
    colorlinks,
    citecolor=black,
    filecolor=black,
    linkcolor=black,
    urlcolor=black
}
\def\dd{\mathrm{d}}

\def\eps{\varepsilon}

\def\Z{\mathbf{Z}}
\def\N{\mathbf{N}}

\def\R{\mathbf{R}}

\def\Hom{\mathrm{Hom}}

\def\Id{\mathrm{Id}}
\def\into{\hookrightarrow}
\def\onto{\twoheadrightarrow}

\def\la{\lambda}

\def\KK{\mathcal{K}}
\def\LL{\mathbbm{L}}
\def\PF{\mathcal{P}_{\mathrm{f}}}
\def\met1{\mathbf{Met_1}}
\def\pmet1{\mathbf{PMet_1}}
\def\Set{\mathbf{Set}}
\def\Top{\mathbf{Top}}
\def\HoTop{\mathbf{HoTop}}
\def\DDelta{\mathbf{\Delta}}

\def\smet1{\mathbf{sMet_1}}
\def\spmet1{\mathbf{sPMet_1}}
\def\sSet{\mathbf{sSet}}
\def\HosSet{\mathbf{HosSet}}
\def\psSet{\mathbf{psSet}}
\def\Fun{\mathbf{Fun}}
\def\nabla{\triangledown}

\def\Sing{\mathrm{Sing}}
\def\PA{\mathrm{P}} 

\date{October 2, 2022.}

\begin{document}
\centerline{}

\title{Randomized simplicial sets}
\author[I.~Marin]{Ivan Marin}
\address{LAMFA, UMR CNRS 7352, Universit\'e de Picardie-Jules Verne, Amiens, France, and IMJ-PRG, UMR CNRS 7586, Paris.}
\email{ivan.marin@u-picardie.fr}
\email{marin@imj-prg.fr}
\medskip

\begin{abstract} We construct new geometric realizations of simplicial and pre-simplicial
sets where the standard $n$-simplex, viewed as the space of probability
measures on $n+1$ elements, is replaced by the space of $(n+1)$-valued random variables,
with the topology of probability convergence. We
prove that the map which associates to a random variable its probability law
is an homotopy equivalence from these new geometric realizations to
the classical ones. Finally, we prove that this realization provides a new Quillen equivalence
between simplicial sets and topological spaces.
\end{abstract}

\maketitle

\tableofcontents

\section{Introduction and main results}

\subsection{Context}
We continue the exploration of Simplicial Random Variables, as
initiated in \cite{SRV}. The observation at the starting point of \cite{SRV} was
that the usual geometric realization $|\KK|$ of a simplicial complex $\KK$
is given by the collection of all the probability measures on a vertex set $S$
whose support provides a face of the simplicial complex. Of course
the usual topology of this geometric realization is the weak topology,
which is not that natural in the realm of measure theory, but Dowker's theorem
tells us that the choice of topology is not that relevant in terms
of homotopy theory, and that one can choose a metric topology instead. Then, from the viewpoint of probability theory,
a natural object living above this geometric realization is the 
(metric) space $L(\KK)  \subset L^1(\Omega,S)$ of random variables with values in the vertex set $S$
whose essential image is a face of $\KK$. Here $(\Omega,\la)$ is an atomless (complete) probability space, the chosen metric on $L^1(\Omega,S)$
is $d(f,g) = \int d(f(t),g(t))\dd \la (t)$, where $S$ is endowed with the discrete metric
$d(x,y) = 1- \delta_{x,y}$. When $S$ is finite, the underlying topology of $L^1(\Omega,S)$
corresponds to the concept of convergence in probability of a sequence of random variables.
We proved in \cite{SRV} that this space has the same homotopy type as the ordinary realization, and that the natural
`probability law' map $L(\KK) \to |\KK|$ is a Serre fibration and a homotopy equivalence. Therefore these spaces of random variables provide
alternative constructions for the geometric realization of the simplicial complex $\KK$.

These alternative constructions have the following merit. 
Given a vertex set $S$, and $\KK$ a simplicial complex with vertices inside
$S$, then a free action of group $G$ on $S$ does not in general
induce a free action on $|\KK|$, but it \emph{does} provide a free action on
$L(\KK)$.

As a consequence, if $G$ is an arbitrary group, then one can consider with new eyes
the obvious candidate for
a universal simplicial complex being acting upon by $G$, the full collection
$\KK_G = \mathcal{P}_{\mathrm{f}}^*(G)$ of non-empty finite subsets of $G$.
The induced action of $G$ on $|\KK_G|$ is not free, thus preventing the construction of
a classifying space for $G$ as $|\KK_G|/G$. But the induced action of $G$
on $L(\KK_G)$ \emph{is} free, providing
an easy construction of a classifying space for $G$, as $L(\KK_G)/G = L^1(\Omega,G)/G$. The
properties of this construction have been studied separately in detail in \cite{CCS}.

\subsection{Constructions and results}

Here we consider the other
standard concept of simpliciality,
namely simplicial sets. Replacing in each case the $n$-simplex $\Delta_n$,
again considered as a space of probability measures, by the space
of $\nabla_n$ of random variables with values in $\{0,\dots,n \}$ yields new realizations of these simplicial
sets as spaces of simplicial random variables.

In order to be more precise, we first recall the basic concepts in the realm of simplicial sets.
We let $\Top$ be a convenient category of topological spaces containing the metrizable ones,
for instance the category of weakly Hausdorff and compactly generated topological
spaces, and $\Set$ the category of sets. In our conventions, compact (and paracompact) spaces are Hausdorff.
We denote $\DDelta$ the category with objects
the $[n] = \{ 0, \dots, n \}, n \in \Z_{\geq 0}$ and (weakly) increasing maps as morphisms.
A simplicial set
is a contravariant functor $F \in \Fun(\DDelta^{op},\Set)$, or equivalently
a graded set $F = \bigsqcup_{ n \geq 0} F_n$ with $F_n = F([n])$, equiped with an action on the right
of the category $\DDelta$. The elements of $F_n$ are called the $n$-simplices of $F$.

The usual geometric realization
has been defined by Milnor \cite{SSMILNOR} as follows.
Let  $\Delta_n = \{ (x_0,\dots,x_n) \in [0,1]^n ; 
\sum_i x_i = 1 \}$ endowed with the product topology of $[0,1]^n$. It defines a (covariant) functor $\Delta : \DDelta \to \Top$ via $\Delta([n]) = \Delta_n$
and, for $\sigma \in \Hom_{\DDelta}([n],[m])$, 
$$
\Delta(\sigma) : (x_0,\dots,x_n) \mapsto \left( \sum_{j \in \sigma^{-1}(i)} x_j \right)_{i=0,\dots,m}.
$$
From this functor, the geometric realization $|F|$ is classically defined as the quotient space of
$E = \bigsqcup_n \left( F_n \times \Delta_n\right) 
$
with $F_n = F([n])$ considered as a discrete topological space, by
 the equivalence relation $\sim$ generated by the
 relations $(\alpha \sigma,a) \sim (\alpha,\Delta(\sigma)(a))$, for $\sigma \in \DDelta$. It is a functor $|.| : \sSet \to \Top$ admitting for right adjoint
 the singular functor $X \mapsto \Sing(X)$ with $\Sing(X)_n$ equal to the
 set of maps $\Delta_n \to X$.  It can be seen as the colimit of the functor $\Delta \circ D_F : C_F \to \DDelta$,
 where $C_F$ is the category of simplices of $F$  (see \cite{FRPIC} \S 4.2)  and $D_F : C_F \to \DDelta$ the forgetful functor.

Now set $\nabla_n = L^1(\Omega,[n])$ considered as a (paracompact)
topological space, with topology defined by convergence in probability,
or equivalently as the underlying topology of the $L^1$ metric. 
We introduce the probability-law map $p_n : \nabla_n \to \Delta_n$,
mapping $f \in \nabla_n$ to $(\la(f^{-1}(i)))_{i=0,\dots,n} \in \Delta_n$.
The first statement suggesting that these concepts of probability theory
are well adapted to the topological simplicial context is the following one.

\begin{proposition} \label{prop:ptransfonat} The map $n \mapsto \nabla_n$ extends to
a functor $\DDelta \to \Top$, and the $(p_n)_{n \geq 0}$ define a natural
transformation $\nabla \leadsto \DDelta$.
\end{proposition}
\begin{proof} For $\sigma : [n] \to [m]$,
the map $\nabla(\sigma)$ maps $f \in \nabla_n = L^1(\Omega,[n])$
to $\sigma \circ f \in L^1(\Omega,[m]) = \nabla_m$.
Moreover, for $f,g \in \nabla_n$, we have
$$
d\left(\nabla(\sigma)(f),\nabla(\sigma)(g)\right)
=
\int d\left(\sigma\circ f(t),\sigma\circ g(t)\right)
\leqslant
\int d\left( f(t), g(t)\right) = d(f,g)
$$
which proves that $\nabla(\sigma)$ is 1-Lipschitz and in particular
continuous. The property $\nabla(\sigma \circ \tau) = \nabla(\sigma) \circ \nabla(\tau)$
is clear, hence $\nabla$ defines a functor $\DDelta \to \Top$.

In order to prove that $n \mapsto p_n$ is a natural transformation, we
need to compare the elements
$p_m \circ \nabla(\sigma)(f)$ and $\Delta(\sigma) \circ p_n(f)$
of $\Delta_m$
for $\sigma \in \Hom_{\DDelta}([n],[m])$ and $f \in \nabla_n$. For $i \in [m]$,
we have
$$
(p_m \circ \nabla(\sigma)(f))_i =
\la((\sigma\circ f)^{-1}(i)) =
\la(f^{-1}(\sigma^{-1}(i))) = \sum_{j \in \sigma^{-1}(i)} \la(f^{-1}(j))
$$
while
$$
(\Delta(\sigma) \circ p_n(f))_i = \Delta(\sigma)((\la(f^{-1}(j)))_{j=0,\dots,n})_i
= \sum_{j `\in \sigma^{-1}(i)} \la(f^{-1}(j))
$$
and this proves the claim.
\end{proof}

From this, the realization $L(F)$ of $F$ as a random variable space functorially
associates to $F$
the quotient of $\bigsqcup_n \left( F_n \times \nabla_n\right)$
by the equivalence relation $\sim$ generated by the
 relations $(\alpha \sigma,a) \sim (\alpha,\nabla(\sigma)(a))$, for $\sigma \in \DDelta$. Although $\nabla_n$ is \emph{not} locally compact, the topology of $L(F)$
has the same degree of tameness as $|F|$ : it is paracompact, compactly generated
and perfectly normal (see Proposition \ref{prop:LFtame}).  
As in the classical case, $L(F)$ can be seen as the colimit of the
functor $\nabla \circ D_F : C_F \to \Top$. 
 By the general machinery (see e.g. \cite{MAY}, \cite{KAN}) this functor $L : \sSet \to \Top$
obviously admits a right adjoint $X \mapsto \Sing_{RV}(X)$
 with $\Sing_{RV}(X)_n$ equal to the set of maps $\nabla_n \to X$.
 
  Moreover, the natural transformation $p : \nabla \leadsto \Delta$
 immediately provides a map $p_F : L(F) \to | F |$ and commutative diagrams
$$
\xymatrix{
\Hom_{\Top}(|F|,X) \ar[d] \ar@{<->}[r] &\Hom_{\sSet}(F,\Sing X) \ar[d] \\
\Hom_{\Top}(L(F),X) \ar@{<->}[r] &\Hom_{\sSet}(F,\Sing_{RV} X) \\
} 
$$ 

Our first main result is the following one. 

\begin{theorem} (see Theorem \ref{theo:LwFw} and Section \ref{sect:kan}) For $F$ a simplicial set, the probability-law map $p_F : L(F) \to |F|$
is an homotopy equivalence. In particular, $L(F)$ has the homotopy type of a CW-complex. For $X$ a topological
space, $\Sing_{RV} X$ is a Kan complex.
\end{theorem}

Let $\mathbf{M}$ denote the subcategory of $\DDelta$
such that $\Hom_{\mathbf{M}}([n],[m])$ is the set of injective applications
inside $\Hom_{\DDelta}([n],[m])$. The elements of $\Hom_{\mathbf{M}}([n],[m])$
are called face maps. 

The elements of $\psSet = \Fun(\mathbf{M}^{op},\Set)$ are called pre-simplicial sets (other common terminologies: semi-simplicial sets, $\Delta$-sets).
In particular, to each
pre-simplicial set $F$ can be associated a geometric realization $\| F \|$ defined as
the quotient of $\bigsqcup_n \left( F_n \times \Delta_n\right) $
 by
 the equivalence relation $\sim$ generated by the
 relations $(\alpha \sigma,a) \sim (\alpha,\Delta(\sigma)(a))$, for $\sigma \in \mathbf{M}$.

As before, we can construct the quotient $\LL(F)$ of 
$\bigsqcup_n \left( F_n \times \nabla_n\right) $ by the equivalence relation generated
by $(\alpha \sigma,a) \sim (\alpha,\nabla(\sigma)(a))$, for $\sigma \in \mathbf{M}$.
Again, the natural transformation $p : \nabla \leadsto \Delta$
 immediately provides a probability-law map $p_F : \LL(F) \to \| F \|$. Our second main result is the
 following one. 

\begin{theorem} (see Theorem \ref{theo:semiLwFFw}) For $F$ a pre-simplicial set, the probability-law map $\LL(F) \to \| F\| $ is an homotopy equivalence.
In particular, $\LL(F)$ has the homotopy type of a CW-complex.
\end{theorem}

In this case, we are able to construct an \emph{explicit} homotopy,
which we will use in exploring the homotopic properties of the construction on simplicial
sets (for instance in the proof of Theorem \ref{theointro:QuillenEq} below).

Finally, recall that if $\KK$ is a simplicial complex over a \emph{totally ordered} set $S$ of vertices, then one can associate to it 
a simplicial set $S\KK$ 
and a pre-simplicial set $M \KK$ (see Section \ref{sect:orderedSC}). 
The next result says that the constructions of this paper
are compatible with the constructions of \cite{SRV} in this case.

\begin{theorem}\label{theo:comprealKK} (see Section \ref{sect:orderedSC})
Let $\KK$ be an ordered simplicial complex. Then $L(S \KK)$ and $L(M \KK)$ are homeomorphic,
and
$L(S \KK)$, $L(\KK)$ and $|\KK|$ are homotopically equivalent.
\end{theorem}

Comparing the above results with the results of \cite{SRV} for simplicial
complexes, this suggests the following conjecture.

\begin{conjecture} The maps $L(F) \to |F|$
and $\LL(F) \to \| F \|$
Hurewicz fibrations. 
\end{conjecture}

We then prove that this new adjunction $L : \sSet \leftrightarrows  \Top : \Sing_{RV}$ is well-behaved with respect to homotopy. Our first result in this direction
compares the homotopy types of $\Sing X$ and $\Sing_{RV} X$.

\begin{theorem}\label{theointro:eqSingSingRV} (see Theorem \ref{theo:eqSingSingRV})
For $X$ a topological space, the natural map $|\Sing X| \to |\Sing_{RV} X|$ is an homotopy equivalence.
In particular, $|\Sing_{RV} X|$ and $X$ have the same weak homotopy type.
\end{theorem}

Thus the natural map $\Sing X \to \Sing_{RV} X$ is a weak homotopy
equivalence inside $\sSet$. If $\Sing X$ is viewed as the $\infty$-groupoid $\pi_{\leq \infty} X$ of $X$, then $\Sing_{RV} X$ provides another construction
for it. More concretely, this Theorem implies in particular that the induced
morphism $\Z \Sing X \to \Z \Sing_{RV} X$ of simplicial abelian groups
is also a weak equivalence (see \cite{GOERSSJARDINE} Proposition III 2.16), so that the obvious `randomized singular
chain complex' of $X$, constructed in the same way as the classical singular chain complex
by replacing the collection of all maps $\Delta_n \to X$ with the
collection of all maps $\nabla_n \to X$,
has for homology the classical singular homology of $X$.

Then, 
we endow
$\Top$ with M. Cole's more flexible version of the standard (Quillen) model category structure. The homotopy equivalences for this structure are the usual weak homotopy equivalences.
 The
classical functors $|\bullet|$ and $\Sing$ together provide a Quillen equivalence between this structure
and the standard model structure on $\sSet$. We get the following `randomization' of this classical result.

\begin{theorem} \label{theointro:QuillenEq} (see Theorem \ref{theo:QuillenEq})
The functors $L : \sSet \to \Top$ and $\Sing_{RV} : \Top \to \sSet$ provide a Quillen equivalence
between $\sSet$ and $\Top$. In particular they induce an equivalence of categories between
the corresponding homotopy categories.
\end{theorem}

{\bf Acknowledgements.} I thank S. Douteau
and D. Chataur
for useful discussions. I thank especially D. Chataur for his help in the proof of Theorem \ref{theointro:eqSingSingRV}.

\section{Preliminaries on measure algebras}
\label{sect:prelimmeasure}

In the paper, $\Omega$ is a atomless (complete) probability space.
It admits a \emph{measure algebra} $\mathfrak{M}$, defined (see \cite{FREMLIN}),
as the collection of all measurable sets modulo neglectable ones, with the operations of intersection $\cap$ and symmetric difference $\Delta$, together with the measure map $\la : \mathfrak{M} \to [0,1]$. It is naturally endowed with a metric $d(X,Y) =\la(X\Delta Y)$,
so that as a metric space it is naturally isomorphic to $\nabla_2$.
The atomless condition implies, thanks to Sierpinsky's theorem (\cite{SIERPINSKY}),
that there exists an \emph{exhaustion map} $t \mapsto \Omega_t$, which is a continuous
map $[0,1] \to \mathfrak{M}$ such that $t_1 \leq t_2 \Rightarrow \Omega_{t_1} \subset \Omega_{t_2}$
and $\la(\Omega_t) = t = t \la(\Omega)$. We fix this exhaustion map once and for all. When $\Omega$ is a \emph{standard} probability space,
one can identify $\Omega$ with $[0,1]$ endowed with the Lebesgue measure and set
$\Omega_t = [0,t]$.

The following useful technical results were proven in \cite{SRV}, under the unnecessary
assumption that $\Omega$ is standard. More generally, all the results of \cite{SRV} are true
without this assumption, with essentially the same proofs (with the exhaustion map replacing
the choices of intervals).
 The suspicious reader may however impose this additional
assumption that $\Omega$ is standard on the current paper as well. Hopefully the detailed proofs for the statements of \cite{SRV} in this more general setting will appear in \cite{SRVbook}.

The first useful map constructed in \cite{SRV} is the following one.
It is
a continuous map $\mathbf{g} : \mathfrak{M} \times [0,1] \to \mathfrak{M}$
such that
\begin{enumerate}
\item for every $A\in \mathfrak{M}$, $u \in [0,1]$, $\mathbf{g}(A,0) = A$, $\la(\mathbf{g}(A,u)) = \la(A)(1-u)$
\item for every $A\in \mathfrak{M}$, $0 \leq u \leq v \leq 1$, $\mathbf{g}(A,u) \supset \mathbf{g}(A,v)$
\item setting $\mathbf{\check{g}}(A,u) = \mathbf{g}(A,1-u)$, 
so that $\la(\mathbf{\check{g}}(A,u)) = u \la(A)$,
we have $\mathbf{\check{g}}(A,uv) = 
\mathbf{\check{g}}(\mathbf{\check{g}}(A,u),v)$ for every $A\in \mathfrak{M}$ and $u,v \in [0,1]$.
\item for every $u,v \in [0,1]$, $\mathbf{g}(\Omega_v,u) = \Omega_v \setminus \Omega_{uv}$
\item for all $E,F \in \mathfrak{M}$ and $u,v \in [0,1]$,
$$\la\left(\mathbf{g}(E,u) \Delta
\mathbf{g}(F,v) \right)\leq 4 \la(E \Delta F)  + |v-u|\max(\la(E),\la(F))
\leq 4 \la(E \Delta F)  + |v-u|$$
\end{enumerate}
This map is constructed in \cite{SRV}, Lemma 6. The additional statements we provide here
are proven in the course of the proof of the Lemma given there.

From this map, we can immediately build two other useful ones.
\begin{enumerate}
\item Setting $\mathbf{h}(A,u) = \, ^c \mathbf{g}( \, ^c A,u)$, one gets a continuous companion map
$\mathbf{h} : \mathfrak{M} \times [0,1] \to \mathfrak{M}$
such that $\mathbf{h}(A,0) = A$, $\mathbf{h}(A,1) = \Omega$, $\la(\mathbf{h}(A,u)) = u + (1-u) \la(A)$ and $\mathbf{h}(A,u) \subset \mathbf{h}(A,v)$ for all $A$ and $u \leq v$. Moreover it satisfies $\mathbf{h}(\Omega_t,u)
= \Omega_{u+(1-u) t}$ and
$$\la\left(\mathbf{h}(E,u) \Delta
\mathbf{h}(F,v) \right)\leq 4 \la(E \Delta F)  + |v-u|$$ for all $E,F \in \mathfrak{M}$ and $u,v \in [0,1]$.
\item The map $(t,A) \mapsto t A = \mathbf{\check{g}}(A,t)$ provides a topological
`retracting' action of the monoid $[0,1]$ (for the multiplication law) on $\mathfrak{M}$, that is
$t_1(t_2 A) = (t_1t_2) A$, with the property that $\la(t A) = t\la(A)$.
\end{enumerate}

A third, more elaborate map is constructed in \cite{SRV} from $\mathbf{g}$.
For $X$ a topological space, let us denote $\PA(X)$ the \emph{path space} made of continuous maps $[0,1] \to X$ endowed with the compact-open topology. Then, there is a continuous map
$$
\Phi : \PA([0,1]) \times \PA(\mathfrak{M}) \times \mathfrak{M} \to \PA(\mathfrak{M})
$$
mapping $(q,E_{\bullet},A)$ to $B_{\bullet}$, so that
\begin{itemize}
\item if $A \subset E_0$ and $q(0)\la(E_0)  = \la(A)$, then $B_0 = A$
\item for all $u \in [0,1]$, $B_u \subset E_u$ and $\la(B_u) = q(u) \la(E_u)$
\item if $q$ and $E_{\bullet}$ are constant maps, then so is $B_{\bullet}$
\end{itemize}
Informally this says that, when $E_{\bullet} \in \PA(\mathfrak{M})$ is a path inside $\mathfrak{M}$ with $A \subset E_0$, then we can find another path $B_{\bullet} \in \PA(\mathfrak{M})$
such that $B_u \subset E_u$ for every $u$, and the ratio $\la(B_{\bullet})/\la(E_{\bullet})$ follows any previously specified variation starting at $\la(A)/\la(E_0)$ -- and, moreover, that this can be done continuously.

The map is constructed as follows. We extend by constants the map $\check{\mathbf{g}} : \mathfrak{M} \times [0,1] \to \mathfrak{M}$ so that to define a continuous map $\mathfrak{M} \times \R \to \mathfrak{M}$.
We have $\check{\mathbf{g}}(A,t) = \check{\mathbf{g}}(A,1) = A$ for every $t \geq 1$,
and $\check{\mathbf{g}}(A,t) = \check{\mathbf{g}}(A,0) = \emptyset$ for every $t \leq 0$. Then,
setting $a(u) = q(u)\la(E_u)$, the image of $(q,E_{\bullet},A)$ is defined by the formula
$$
u\mapsto \check{\mathbf{g}}\left( A \cap E_u, \frac{a(u)}{\la(A \cap E_u)}\right) \cup
\check{\mathbf{g}}\left( E_u \setminus A, \frac{a(u)-\alpha(u)}{\la(E_u\setminus A)} \right)
$$
A detailed elementary proof that this map is indeed continuous can be found in \cite{SRV} (see Proposition 5 there).

Finally, we notice that the topological space $\nabla_n$ depends
only on the measure algebra $\mathfrak{M}$, as it can be defined
as 
$$\nabla_n = \{ \underline{A}=(A_k)_{k=0..n} \in \mathfrak{M}^{n+1} \ | \ 
i \neq j \Rightarrow A_i \cap A_j = 0 \ \& \ \sum_{k=0}^n \la(A_k) = 1  \}
$$
the correspondance with the description of $f \in \nabla_n$ as a map $f : \Omega \to [n]$
being given by $A_k = f^{-1}(k)$. 
Therefore, our construction $L(F)$ depends only on $\mathfrak{M}$, and not
on the probability space $\Omega$ itself. As a consequence, from
Maharam's theorem (see e.g. \cite{FREMLIN} ch. 33), we could assume w.l.o.g. that $\Omega$ is a countable
union of (renormalized) probability spaces of the form $\{ 0, 1\}^{\alpha}$
with the $\alpha$ infinite cardinals. We shall not need this fact, though, in
the course of our proofs.

Another remark is that all the constructions made here make sense if $\mathfrak{M}$ is replaced
by \emph{any} subalgebra of $\mathfrak{M}$ containing the subsets $\Omega_t, t \in [0,1]$ : typically, from
the construction in \cite{SRV}, we get immediately that the set $\mathbf{g}(A,u)$ belongs to the subalgebra generated by the $\Omega_t$ and $A$. As an example of such a subalgebra, in the case
where $\mathfrak{M} = \mathfrak{M}([0,1])$ is the measure algebra of the unit interval
and the exhaustion map is the map $t \mapsto [0,t]$,
the algebra $\mathfrak{M}$ could be replaced by any subalgebra containing the unions
of any finite number of open (or closed) intervals of $[0,1]$. Therefore, all the results
of the present paper remain valid if, in the above definition of $\nabla_n$, the algebra $\mathfrak{M}$ is replaced by any of these subalgebras.

\section{Simplicial Random Variables}
\label{sect:SRV}

The purpose of this section is to prove the following theorem.

\begin{theorem} \label{theo:LwFw}
The probability-law map $L(F) \to |F|$ is a homotopy equivalence.
\end{theorem}

In order to prove it, we first need to prove that $L(F)$ has
similar structural properties as $|F|$, so we need to
browse the proofs describing the structure of $|F|$ as they
appear in standard textbooks and prove that they can be adapted to $L(F)$
(without, in particular, using the theory of CW-complexes).
We use \cite{FRPIC} for this purpose throughout.

We fix some $F \in \sSet$, and let $F_n^{\#} \subset F_n$ the collection of non
degenerate simplices, that is the ones not inside $F_{n-1}.\sigma$ for some $\sigma \in \DDelta$. Recall that $|F|$ is a quotient
of the subspace $\bigsqcup F_n^{\#}\times \Delta_n$ (\cite{FRPIC} cor. 4.3.2) and that it is
a CW-complex (\cite{FRPIC}, Theorem 4.3.5), 
the $n$-cells being given by the $\{ x \} \times \Delta_n \simeq \Delta_n \simeq B^n$ for $x \in F_n^{\#}$
with attaching maps $c_x : \Delta_n \to |F|$ mapping $u \in \Delta_n$ to the class of $(x,u)$.

For technical purposes, we need to introduce standard subcategories of $\DDelta$.
We let $\mathbf{M}$ (resp. $\mathbf{E}$) denote the subcategory of $\mathbf{\Delta}$
such that $\Hom_{\mathbf{M}}([n],[m])$ (resp. $\Hom_{\mathbf{E}}([n],[m])$) is the set of injective (resp. surjective) applications
inside $\Hom_{\DDelta}([n],[m])$. The elements of $\Hom_{\mathbf{M}}([n],[m])$
are called the face maps and the elements of $\Hom_{\mathbf{E}}([n],[m])$ are called
the degeneracy maps.

We first notice that the composition of the functor $\nabla : \DDelta \to \Top$
with the forgetful functor $\Top \to \Set$ provides a cosimplicial
set, which is immediately checked to have the Eilenberg-Zilber property.
Recall from e.g. \cite{FRPIC} Proposition 4.2.6 that this property means 
 $\forall x \in \nabla_0 \ \nabla(\delta_0)(x) \neq \nabla(\delta_1(x))$
 with $\{ \delta_0,\delta_1 \} = \Hom_{\DDelta}([0],[1])$,
and has for consequence that every element of $L(F)$ admits a unique
representative of the form $(\alpha,a)$ with $\alpha \in F_n^{\#}$
and $a$ an interior point of $\nabla_n$ (that is, a point not inside the image
of $\nabla(\sigma)$ for $\sigma : [m] \to [n]$, $m < n$). This representative
is called the minimal representative.

\subsection{The boundary and interior of $\nabla_n$}

Recall from \cite{SRV} that to every simplicial complex $\KK$
one associates the metric space $L(\KK)$, defined as a subspace
of $L^1(\Omega,S)$ for $S = \bigcup \KK$ the union of all the elements of $\KK$, that is its vertex set. This subspace is made of the
(up to neglectability, measurable) maps 
$f : \Omega \to \bigcup \KK$ such that $f(\Omega) \in \KK$,
where
$$
f(\Omega) = \{ s \in S \ | \ \la(f^{-1}(s)) > 0 \}
$$
is what is called the \emph{essential image} of $f$. In
this context, 
$\Delta_n = |\PF^*([n])|$ and 
$\nabla_n = L(\PF^*([n]))$, where $\PF^*(S)$ denotes the collection
of all nonempty finite subsets of the set $S$.

For any
simplicial complex $\KK$, let us consider the
set $\KK_{\mathrm{max}}$ of its maximal elements. Then it is easily checked that $\partial \KK = \KK \setminus \KK_{\mathrm{max}}$ is a simplicial complex as well.
We set $\partial \nabla_n = L(\partial \PF^*([n]))= L(\PF^*([n]) \setminus \{ [n] \})$
and 
$$\nabla_n^{\circ} = \nabla_n \setminus \partial \nabla_n
= \{ f \in L^1(\Omega,[n]) ; \ f(\Omega) \subsetneq [n]\}.$$ 
It is easily checked
that $\nabla_n^{\circ}$ is equal to the interior of $\nabla_n$ in the sense
of the cosimplicial set $\nabla$ as defined above.

We can now notice the following properties of $\nabla$, whose easy proofs are left to the reader. Part (3) can be proved either directly
or, applying $p : \nabla \leadsto \Delta$, immediately deduced from the classical case.

\begin{lemma} \label{lem:propsnabla} \ 
\begin{enumerate}
\item For every $S \in \mathbf{E}$, the map $\nabla(S)$ maps interior points to interior points.
\item For every $D \in \mathbf{M}$, the map $\nabla(D)$ is a closed map.
\item If $\sigma \in \Hom_{\DDelta}([m],[n])$ and $a \in \nabla_m^{\circ}$, then
$\nabla(\sigma).a$ determines $\sigma$.
\end{enumerate}
\end{lemma}

Together with the Eilenberg-Zilber property, part (1) of the lemma has the following easy consequence.
Let $(\beta,b) \in \bigsqcup F_n \times \nabla_n$ having $(\alpha,a)$
for minimal representative. We have $b = \nabla(D).a'$ for some $D \in \mathbf{M}$ and some interior
point $a'$ by the
Eilenberg-Zilber property, hence $(\beta,b) \sim (\beta.D,a')$; now, $\beta.D= \alpha'.S$ for some $S \in \mathbf{E}$
and $\alpha' \in F^{\#}$, so that $(\beta.D,a') = (\alpha'.S,a') \sim (\alpha',\nabla(S).a')$. But since $a'$
is an interior point so is $\nabla(S).a'$ and $(\alpha',\nabla(S).a')$ is the unique minimal representative,
which proves $\alpha' = \alpha$, $a = \nabla(S).a'$. In particular we have $D \in \mathbf{M}$ and $S \in \mathbf{E}$
such that $\beta .D = \alpha.S$.

 By \cite{SRV} we know that the probability-law map
provides a Hurewicz fibration $\partial \nabla_n = L(\partial \PF^*([n])) \to |\partial \PF^*([n])| = \partial \Delta_n$
with homotopically trivial fiber (consider the preimage of a vertex of $\partial \Delta_n$), which is an homotopy equivalence. Since $\Delta_n$ is homeomorphic to a $n$-sphere, we get
that $\partial \nabla_n$ has the (strong) homotopy type of a $n$-sphere. Moreover, $\partial \nabla_n$ is equal to the preimage of $\partial \Delta_n$ under the probability-law map.

\subsection{The cofibration $\partial \nabla_n \to  \nabla_n$.}

The purpose of this section is to prove the following.
\begin{proposition} \label{prop:cofibnabla} The inclusion map $\partial \nabla_n \to  \nabla_n$ is a closed cofibration.
\end{proposition}
As $\partial \nabla_n$ is a closed subset of $\nabla_n$, in order to prove the
proposition we need (see e.g. \cite{FRPIC}, Proposition A.4.1 p.250) to construct a retract $\nabla_n \times I \to (\nabla_n \times \{ 0 \})\cup ((\partial \nabla_n) \times I)$ of the reverse natural inclusion,
with $I = [0,1]$.

We first follow the classical geometric receipe (\cite{FRPIC}, p. 7-8) for proving that
the inclusion of the $n$-sphere inside the $(n+1)$-ball is a closed fibration, except that we do it
on the $n$-simplex (see Figure \ref{fig:cofibDelta}). For
this we construct
a retract $\Delta_n \times I \to (\Delta_n \times \{ 0 \})\cup ((\partial \nabla_n) \times I)$ of the natural inclusion in the other direction.
 The elements of $\Delta_n \times I $ are
the $(\underline{u};a) \in \Delta_n \times I$ for $\underline{u} = (u_0,\dots,u_n)$
with $u_i \geq 0$ and $\sum u_i = 1$. The line from $(\underline{v};2)$  to $(\underline{u};a)$
with $\underline{v}=(v_i)_{i=0..n}$ and $v_i= \frac{1}{n+1}$
 crosses $(\Delta_n \times \{ 0 \})\cup ((\partial \Delta_n) \times I)$
 at exactly one point. The corresponding (continuous) map from $\Delta_n \times I $ is explicitely given by the following formulas
 $$
\begin{array}{lcll}
(\underline{u};a) &\mapsto & ( \frac{1}{2-a} (2u_i-  \frac{a}{n+1})_{i=0,\dots,n};0) & \mbox{ if } a \leq 2(n+1) m(\underline{u})\\
& & (  (\frac{u_i - m(\underline{u})}{1-(n+1)m(\underline{u})})_{i=0,\dots,n}; \frac{a - 2(n+1)m(\underline{u})}{1 - (n+1)m(\underline{u})} )&
\mbox{ if } a \geq 2(n+1) m(\underline{u})
\end{array}
$$
where $m(\underline{u}) = \min(u_0,u_1,\dots,u_n)$.

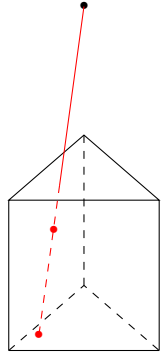
\begin{figure}

\begin{center}

\begin{tikzpicture}[scale=2]
\draw (0,0) -- (1,0); 
\draw[dashed] (0,0) -- (0.5,0.4330127018922193) -- (1,0);
\draw (0,1) -- (1,1) -- (0.5,1.4330127018922193) -- cycle;
\draw (0,0) -- (0,1);
\draw (1,0) -- (1,1);
\draw[dashed] (0.5,1.4330127018922193) -- (0.5,0.4330127018922193);
\draw (0.5,2.28867513459481287) node {$\bullet$};
\draw[red] (0.5,2.28867513459481287) -- (0.338,1.1);
\draw[red,dashed] (0.338,1.1) -- (0.2,0.1);
\draw[red] (0.2,0.1) node {$\bullet$};
\draw[red] (0.3,0.8) node {$\bullet$};
\end{tikzpicture}

\end{center}
\caption{The cofibration $\partial \Delta_n \into \Delta_n$}
\label{fig:cofibDelta}
\end{figure}

We now want to lift the map $\Delta_n \times I \to (\Delta_n \times \{ 0 \})\cup ((\partial \Delta_n) \times I)$
to a map $\nabla_n \times I \to (\nabla_n \times \{ 0 \})\cup ((\partial \nabla_n) \times I)$.
For this we use the following result from \cite{SRV} (Proposition 4.4 and Remark 4.5).

\begin{proposition}\label{prop:HLPplus} Let $X$ be a topological space. Then the probability-law map $p_n : \nabla_n \to \Delta_n$
has the homotopy lifting property w.r.t. $X$, that is, for any (continuous) maps $H : X \times [0,1] \to \Delta_n$,
$h : X \to \nabla_n$ such that $p_n \circ h = H(\bullet,0)$, there exists a map $\tilde{H} : X \times [0,1] \to \nabla_n$ such that $p_n \circ \tilde{H} = H$ and $\tilde{H}(\bullet,0) = h$. Moreover, for any $x \in X$ such that
$H(x,\bullet)$ is constant, then so is $\tilde{H}(x,\bullet)$.
\end{proposition}

We then start from the map
$f : \Delta_n \times [0,1] \to (\Delta_n \times \{ 0 \})\cup (\partial \Delta_n \times I) \subset \Delta_n \times I$
constructed above and we consider the projection map $p_1 : \Delta_n \times I \to \Delta_n$
as well as the composed map $p_1 \circ f = f^1 : \Delta_n \times I \to \Delta_n$. Let us
consider the probability-law map $p_n : \nabla_n \to \Delta_n$ and set
$H= f^1 \circ (p_n \times \Id) : \nabla_n \times I \to \Delta_n$.
We have $H(x,t) = p_1(f(p_n(x),t))$, and $H(x,0) = p_1(f(p_n(x),0)) = p_n(x) = p_n(h(x))$ for
$h = \Id_{\nabla_n}$. Applying Proposition \ref{prop:HLPplus} with $X = \nabla_n$, we get
$\tilde{H} : \nabla_n \times [0,1] \to \nabla_n$ such that $p_n \circ \tilde{H} = H$ and $\tilde{H}(\bullet,0) = h = \Id_{\nabla_n}$. Moreover, for any $x \in \partial \nabla_n$, since $f^1(p_n(x)) = p_n(x)$ we get
that $H(x,\bullet) = f^1(p_n(x),\bullet)$ is constant, since $f^1(y,t) = y$ for all $y \in \partial \Delta_n$.
This yields $\tilde{H}(x,t) = \tilde{H}(x,0) = h(x) = x$ for all $x \in \partial \nabla_n$, $t \in I$.
Let us now consider $\varphi = \pi_2 \circ f : \Delta_n \times [0,1] \to I$ where $\pi_2$ is the second projection
and set $\Psi(x,t) = (\tilde{H}(x,t), \varphi(x,t))$. This defines a continuous map $\Psi : \nabla_n \times I \to \nabla_n \times I$ such that $p_n \times \Id \circ \Psi$ coincides with $f$. As a consequence it takes
values inside 
$$
(\nabla_n \times \{ 0 \})\cup ((\partial \nabla_n) \times I) = (p_n \times \Id)^{-1}\left(
 (\Delta_n \times \{ 0 \})\cup ((\partial \Delta_n) \times I) \right)
$$
and it makes the following diagram commute, where the vertical maps are restrictions of $p_n \times \Id$.
$$
\xymatrix{
\nabla_n \times I \ar[r]\ar[d] & (\nabla_n \times \{ 0 \})\cup ((\partial \nabla_n) \times I)\ar[d]\\
\Delta_n \times I \ar[r] & (\Delta_n \times \{ 0 \})\cup ((\partial \Delta_n) \times I)
}
$$
It remains to prove that $\Psi$ is the identity both on 
$\nabla_n \times \{ 0 \}$, which is clear because
$\Psi(x,0) = (\tilde{H}(x,0),\varphi(x,0)) = (h(x),0) = (x,0)$,
and on $(\partial \nabla_n) \times I$, which holds true because
$\varphi(x,t) = t$ and $\tilde{H}(x,t) = \tilde{H}(x,0)=x$ 
whenever $x \in \partial \nabla_n$. This concludes the proof of Proposition \ref{prop:cofibnabla}.

\subsection{Preliminaries on attachments}

In the remaining part of this section we adopt the point of view of a simplicial
set $G$ as a graded set $\bigsqcup_n G_n$ endowed with a right action of the category
$\DDelta$. A simplicial subset of $G$ is a simplicial set $D = \bigsqcup_n D_n$
with $D_n \subset G_n$ such that the inclusion $D \subset F$ is a simplicial map.

We briefly recall the definition of a simplicial attachment (see  \cite{FRPIC} p. 144).
Let $A$ and $G$ be two simplicial sets, and $D$ a simplicial subset of $G$. That is,
$D$ is a simplicial set $\bigsqcup_n D_n$ with $D_n \subset G_n$ and the inclusion maps $D_n \to G_n$ commute
with the face and degeneracy maps. Moreover, let $f : D \to A$ be a simplicial map.
 In order to avoid confusions, we temporarily denote $x\star \rho$ the action of $\rho \in \DDelta$ on $x \in G_n$. From this the simplicial attachement
$F$ is such that $F_n = A_n \sqcup (G_n \setminus D_n)$ and, for any $\rho \in \Hom_{\DDelta}([n],[m])$ and $x \in F_n$, we define $\rho.x$ from the action of $\DDelta$ on $A$ if $x \in A_n$, from the action $\star$ on $G$ if $x \in G_n \setminus D_n$ and $x.\rho \not\in D_m$, and
finally as $f_m(x\star\rho)$ if  $x.\rho \in G_n \setminus D_n$ and $x\star \rho \in D_m$. Checking that this construction is well-defined is straightforward.

Now, we recall from e.g. \cite{FRPIC} Corollary 4.2.4 that the $n$-skeleton $F^n$ of a simplicial set
$F$ is obtained from its $(n-1)$-skeleton by attaching the simplicial set
$\bigsqcup_{x \in F_n^{\#}}\Delta_x$, where $\Delta_x$ is a copy of the simplicial set $\Delta_n$, via the simplicial map $\bigsqcup_{x \in F_n^{\#}}\varphi_x$ with $\varphi_x : \partial \Delta_x \to F^{n-1}$ given by $\varphi_x(\alpha) = x \alpha$ where the simplicial set $\partial \Delta_n$ is by definition the $(n-1)$-skeleton
of $\Delta_n$.

\subsection{Properties of the functor $L : \sSet \to \Top $}

Let $F,G$ two simplicial sets, and $f : F \to G$ a simplicial map. This means
that $f$ is a collection of maps $f_n : F_n \to G_n$ commuting 
with the right action of the category $\DDelta$, in the sense that, for
every $\sigma \in \Hom_{\DDelta}([n],[m])$ and $\alpha \in F_n$, we
have $f_m(\alpha.\sigma) = f_n(\alpha).\sigma$.
It induces continuous maps $\hat{f}_n =  f_n \times \Id_{\nabla_n} : F_n \times \nabla_n \to G_n \times \nabla_n$.  For $\sigma \in \Hom_{\DDelta}([n],[m])$
and $(\alpha,a) \in F_n \times \nabla_n$ we have
$$
\hat{f}_m(\alpha.\sigma,a) = (f_m(\alpha.\sigma),a)
= (f_n(\alpha).\sigma,a) \sim (f_n(\alpha),\nabla(\sigma)(a)) = \hat{f}_n(\alpha,\nabla(\sigma)(a))
$$
hence $\bigsqcup_n \hat{f}_n$ induces a continuous map $L(f) : L(F) \to L(G)$,
and clearly $L(f \circ g) = L(f) \circ L(g)$, $L(\Id_F) = \Id_{L(F)}$. Therefore $L$
defines a functor $L : \sSet \to \Top$.

The composite of $\nabla$ with the forgetful functor $V : \Top \to \Set$
is a cosimplicial set, and clearly $V \circ L (F) = F \otimes V \circ \nabla$ with the notations
of e.g. \cite{FRPIC}. Moreover, it is immediately checked that $V \circ \nabla$ has the
Eilenberg-Zilber property, and therefore $V \circ L$ preserves and reflects monomorphisms
(\cite{FRPIC}, corollary 4.2.9). In particular, if $D$ is a simplicial subset of $G$,
then the induced map $L(D) \to L(G)$ is injective.

As in the classical case, we have the following property.

\begin{lemma} \label{lem:LGLFclosed}
If $G$ is a simplicial subset of the simplicial set $F$, then the natural
map $L(G) \to L(F)$ embeds $L(G)$ as a closed subset of $L(F)$.
\end{lemma}
\begin{proof}

Let $\bar{y} \in L(G)$ and $\bar{x}$ its image in $L(F)$. There exists unique minimal representatives
of $\bar{x}$ and $\bar{y}$ inside $\bigsqcup_n F_n \times \nabla_n$ and $\bigsqcup_n G_n \times \nabla_n$,
respectively. Since $G^{\#} \subset F^{\#}$ they are the same, and this implies that $\bar{x}$ determines
$\bar{y}$, whence $L(G) \subset L(F)$.

Let now $C$ be a closed subset of $L(G)$, $q : \bigsqcup F_n \times \nabla_n \to L(F)$ the natural
projection map, and $C_{\alpha} = q^{-1}(C) \cap \{ \alpha \} \times \nabla_n$ for each $\alpha \in F_n$.
We need to prove that each $C_{\alpha}$ is closed. This is clear when $\alpha \in G$, so we
assume otherwise, and consider $(\alpha,y) \in C_{\alpha}$. Then $y = \nabla(D).y_0$ for some interior
point $y_0$ and $D \in \mathbf{M}$, and $\alpha.D = \beta.S$ for some $S \in \mathbf{E}$ and $\beta \in F^{\#}$.
Then 
$$(\alpha,y) = (\alpha,\nabla(D).y_0) \sim (\alpha.D,y_0) = (\beta.S,y_0) \sim (\beta,\nabla(S).y_0)
$$
and $\beta$ is non-degenerate, $\nabla(S).y_0$ is interior (Lemma \ref{lem:propsnabla} (1)) , hence $(\beta,\nabla(S).y_0)$
is the minimal representative in the class, which implies $\beta \in G$. Then $\alpha.D = \beta.S  \in G$
and $y = \Delta(D).y_0$ with $(\alpha.D,y_0) \in C_{\alpha.D}$. This implies
$$
C_{\alpha} = \bigcup_{\stackrel{D \in \mathbf{M}}{\alpha.D \in G}} \nabla(D)(C_{\alpha.D}).
$$
Now, each $C_{\alpha.D}$ is closed, each $\nabla(D)$ is a closed map (Lemma \ref{lem:propsnabla} (2)) and the collection
of all $\Delta \in \mathbf{M}$ that can be applied to $\alpha$ is finite, whence $C_{\alpha}$ is closed.
\end{proof}

\bigskip

We show that, when $F$ is a simplicial set, then $L(F)$
can be constructed as a limit of successive attachments. Notice that,
because of Proposition \ref{prop:cofibnabla}, the natural maps $F_n^{\#} \times \partial \nabla_n \to F_n^{\#} \times \nabla_n$
implied in the attachment are closed cofibrations. 

\begin{proposition} \label{prop:LXattachement}
Let $F$ be a simplicial set, then $(L(F^{(n)}))_{n \geq 0}$ is a filtration of $L(F)$ which determines the topology of $F$. Moreover, for every $n$,
$$L(F^{(n)}) = L(F^{(n-1)}) \cup_{L(\varphi)} \left( \bigsqcup_{x \in F_n^{\#}} \nabla_x \right).
$$
where $\varphi : \bigsqcup_{x \in F_n^{\#}} \partial \Delta_x \to F^{(n-1)}$ is the simplicial attaching map, and the natural map $\bigsqcup_n F_n^{\#} \times \nabla_n \to L(F)$ is a quotient map.
\end{proposition}
\begin{proof}
This statement is adapted from the classical analogous statement for the geometric realization functor, so
we need to check that the classical proof uses only properties of the cosimplicial \emph{space}
$[n] \mapsto \Delta_n$ (that we still denote $\DDelta$) that are also satisfied by the cosimplicial
space $[n] \mapsto \nabla_n$ (that we still denote $L$). For this we follow the steps described in \cite{FRPIC}, \S 4.3. First of all, since the cosimplicial set $L$ also has the Eilenberg-Zilber property, then
any element of the tensor product $F \otimes L = L(F)$ can be represented by a unique so-called minimal pair
(\cite{FRPIC}, Proposition 4.2.7). At the set-theoretical level this implies (see \cite{FRPIC}, Proposition 4.3.3) that every
element of $L(F)$ has a unique representative of the form $(\alpha,a)$ with $\alpha$ a non-degenerate
simplex of $F$ of some dimension $n$, and $a \in \nabla_n \setminus \partial \nabla_n$,
and also that, for $f$ a simplicial map, $f$ is injective iff $L(f)$ is injective.
Moreover, the topology of $L(F)$ is the final topology w.r.t. the family
of maps $\overline{c}_x : \nabla_n \to L(F)$, induced by $a \mapsto (x,a) \in F_n \times \nabla_n \to L(F)$
(compare with \cite{FRPIC} p. 153). As a consequence (see \cite{FRPIC}, Proposition 4.3.1) we get that, if
some subset $E \subset \bigsqcup_n F_n$ generates the simplicial set $F$, then $L(F)$ is a quotient space of the subspace $\bigsqcup_n (E \cap F_n) \times \nabla_n$ of $\bigsqcup_n F_n \times \nabla_n$ ; in particular, 
 $L(F)$ is a quotient space of $\bigsqcup_n F_n^{\#} \times \nabla_n$.

From this and Lemma \ref{lem:LGLFclosed} we can adapt
the proof of \cite{FRPIC}, Theorem 4.3.5. First of all, we get from the previous point that the $L(F^{(n)})$ are
closed subspaces of $L(F)$ and form a filtration of it. The fact that the topology of the space $L(F)$ 
is determined by the family $(L(F^{(n)}))_{n \in \N}$ has the same proof as for $|F|$ :
if a map $f : L(F) \to Z$ is such that all its restrictions to $L(F^{(n)})$
are continuous, then so are the composites $f \circ \overline{c}_x$ since any $\overline{c}_x$
factorizes through $L(F^{(n)}) \to L(F)$ for some $n$ ; since $L(F)$ has the final topology
with respect to the $\overline{c}_x$, this fact follows.
 Therefore,
we can fix $n$, and consider $F^{(n)}$ as a simplicial attachment. Then the proof of Theorem 4.3.5 of \cite{FRPIC} can be applied verbatim to our case, and $L(F^{(n)})$ can be described as a topological attachment as in
the statement.
\end{proof}

From this we get the following result. 
\begin{proposition}\label{prop:LFtame} For $F$ a simplicial set, $L(F)$ is paracompact and perfectly normal. It is also compactly generated.
\end{proposition}
\begin{proof}
Recall that a space $X$ is called perfectly normal if every closed subset is the vanishing
locus of some map $X \to \R_+$. Since discrete spaces and the metrizable spaces $F_n^{\#}\times \nabla_n$ are perfectly normal,
by induction
on $n$ we get from Proposition \ref{prop:LXattachement} that all the $L(F^{(n)})$ are perfectly normal (see \cite{FRPIC} Proposition A.4.8 (iv)). Moreover, the embeddings $L(F^{(n)}) \into L(F^{(n+1)})$ are closed
cofibrations (see \cite{FRPIC} Proposition A.4.8 (ii)) and we checked in the proof of Proposition \ref{prop:LXattachement} that the topology
of $L(F)$ is the topology of the union of the $L(F^{(n)})$. From this it follows
(\cite{FRPIC} Proposition A.5.1 (iv)) that $L(F)$ is perfectly normal. In particular it is Hausdorff. In order to prove that it is paracompact, it is by the same arguments
enough to check that each of the $L(F^{(n)})$ is paracompact (\cite{FRPIC} Proposition A.5.1 (v)). By Michael's theorem (\cite{MICHAEL1953}, pages 791-792; \cite{MICHAEL1956}) this
follows by induction on $n$ from Proposition \ref{prop:LXattachement} and the fact that each of the $F_n^{\#} \times \nabla_n$ is metrizable hence paracompact. 

We now prove that $L(F)$ is compactly generated. The fact that
each $L(F^{(n)})$ is compactly generated follows from the proposition
by induction on $n$ (see e.g. \cite{MAYC} ch. 5.2), as
$F_n^{\#} \times \nabla_n$ is (metrizable hence) first countable hence compactly generated. Since the filtration
$L(F^{(n)})$ determines the topology of $L(F)$ this implies that $L(F)$
is compactly generated.

\end{proof}

\begin{remark} Recall that every subset of a space which is both paracompact and perfectly
normal is also paracompact and perfectly
normal (\cite{LW} Appendice I Theorem 6), so this property of CW-complexes is also shared by $L(F)$.
\end{remark}

\subsection{Proof of Theorem \ref{theo:LwFw}}

We follow the scheme of the proof of the comparison theorem of \cite{FRPIC} Theorem 4.3.20, and adapt
it to our case. Let $F$ be a simplicial set. The natural map $p_F^{(0)} : L(F^{(0)}) \to |F^{(0)}|$ is the identity
map on a disjoint union of points, therefore it is a homotopy equivalence. Let us assume that
we know that the probability-law maps $p_F^{(k)} : L(F^{(k)}) \to |F^{(k)}|$ for $k \leq n-1$ are homotopy equivalences and commute with the
natural inclusion maps $|F^{(k-1)}| \subset |F^{(k)}|$ and $L(F^{(k-1)}) \subset L(F^{(k)})$. 
By Proposition \ref{prop:LXattachement} we have a commutative diagram
$$
\xymatrix{
\bigsqcup_{x \in F_n^{\#}}\ar[d]  \nabla_n & \bigsqcup_{x \in F_n^{\#}}\ar[d]\ar[l] \ar[r]   \partial \nabla_n & L(F^{(n-1)})\ar[d]  \\  
\bigsqcup_{x \in F_n^{\#}} \Delta_n & \bigsqcup_{x \in F_n^{\#}}\ar[l] \ar[r]   \partial \Delta_n & |F^{(n-1)}|  \\  
}
$$
where the vertical maps are homotopy equivalences and the horizontal maps going left are closed cofibrations
by Proposition \ref{prop:cofibnabla}. By the gluing theorem (\cite{FRPIC} Theorem A.4.12 and \cite{BROWN} 7.5.7)
this implies that $p_F^{(n)} : L(F^{(n)}) \to |F^{(n)}|$ is a homotopy equivalence. By induction this proves that
$p_F^{(n)}$ is a homotopy equivalence for every $n$. This provides a commutative
ladder of homotopy equivalences $(p_F^{(n)})_{n \geq 0}$, thus the induced map $p_F : L(F) \to |F|$ between the union spaces is a homotopy equivalence (\cite{FRPIC}, Proposition A.5.11). This concludes the proof of the theorem.

\section{Pre-simplicial random variables}
\label{sect:presimpl}

We use \cite{FRPIC} and \cite{EBERTORW}  for reference, and recall the notations
of the introduction. Recall from Section \ref{sect:SRV}
that
$\mathbf{M}$ and $\mathbf{E}$ denotes the subcategories of $\mathbf{\Delta}$ made of the face and degeneracy maps, respectively.
The objects of $\psSet = \Fun(\mathbf{M}^{op},\Set)$ are called pre-simplicial sets
(or semi-simplicial sets, or $\Delta$-sets).

The categories $\mathbf{M}$ and $\mathbf{E}$ are generated
by the elementary face and degeneracy maps, respectively. These are defined
as
$$
\begin{array}{rcrclcrcrclc}
D^c_i &:& [n-1] &\to& [n]& &S^c_i  &: &[n+1] & \to & [n] \\
    & & k & \mapsto & k& \mbox{if $k<i$} & & & k &\mapsto& k& \mbox{if $k\leq i$}\\ 
    & & & & k+1 & \mbox{if $k \geq i$} & & & & & k-1 & \mbox{if $k > i$} \\
\end{array}
$$

 In particular, the geometric realization $\| F \|$ defined in the introduction
 can also be
defined as the quotient of $\bigsqcup_n \left( F_n \times \Delta_n\right)$
by the equivalence relation 
generated by
the $(D_i \alpha,a) \approx (\alpha,D^i a)$ where $D_i = F(D_i^c) : F_n \to F_{n-1}$
and $D^i = \Delta(D_i^c)$. Similarly, $\LL(F)$ is the quotient of
$\bigsqcup_n \left( F_n \times \nabla_n\right)$
by the equivalence relation 
generated by
the $(D_i \alpha,a) \approx (\alpha,D^i_{RV} a)$ where
$D^i_{RV} = \nabla(D_i^c)$.

In this section we will prove the following
\begin{theorem} \label{theo:semiLwFFw}
For every pre-simplicial set $F$, the probability-law map $\LL(F) \to \|F\|$ is an homotopy equivalence.
\end{theorem}

Recall from \cite{SRV} (Proposition 4.1 and its proof) that the probability-law map $p_n : \nabla_n \to \Delta_n$
has a (continuous) section $\sigma_n : \Delta_n \to \nabla_n$ characterized
by 
$$
\sigma_n(\alpha)(x) = a \ \ \mbox{    if    } \ \ x \in \Omega_{\sum_{u\leq a} \alpha(u)} \setminus \Omega_{
\sum_{u<a} \alpha(u)}
$$
for $\alpha : [n] \to [0,1]$ with $\sum_{k=0}^n \alpha(k) = 1$. The main theorem will readily follow
from the following proposition, which provides an explicit homotopy equivalence. This will
be also used in Section \ref{sect:quillen}.

\begin{proposition} \label{prop:commuteface}
The composition $\sigma_n \circ p_n : \nabla_n \to \nabla_n$ is
homotopic to the identity map, by an homotopy $H_n : [0,1] \times \nabla_n \to \nabla_n$
which commutes with
the face maps $D^i_{RV}$, that is
$$
\forall n\, \forall f \in \nabla_n\, \forall i \in [0,n+1] \, \forall u \in [0,1]\, H_{n+1}(u,D^{i}_{RV} f)= D^{i}_{RV} H_n(u,f)
$$
and such that
\begin{itemize}
\item $p_n(H_n(u,f)) = p_n(f)$ for all $f \in \nabla_n, u \in [0,1]$.
\item $H_n(u,\sigma_n(\alpha)) = \sigma_n(\alpha)$ for all $\alpha \in \Delta_n, u \in [0,1]$.
\end{itemize}  
\end{proposition}
\bigskip

The proof of the Theorem then goes as follows. The maps $\sigma_n$, $p_n$ and $H_n$ for $n \geq 0$
together define maps $\tilde{\sigma}_F : A \to B$, $\tilde{p}_F : B \to A$ and $\tilde{H}_F : I \times B \to B$
with $A =\bigsqcup_n F_n \times \Delta_n$ and $B = \bigsqcup_n F_n \times \nabla_n$.
Because of the compatibility with the face maps, they induce maps $\sigma_F : \| F \| \to \LL(F)$,
$p_F :   \LL(F) \to \| F \|$ and $H_F : I \times \LL(F) \to |F|$. It is readily checked
that $\tilde{p}_F \circ \tilde{\sigma}_F = \Id$, hence $p_F \circ \sigma_F = \Id$. By the proposition,
$H_F$ is an homotopy between the identity and $\sigma_F\circ p_F$, and this proves the claim.

In the remaining part of this section, we prove Proposition \ref{prop:commuteface}.

\subsection{Homotopy equivalence : definition in the case $n = 1$}
\label{sect:homotopn1}
\ 

We use the retracting action $(t,A) \mapsto tA$ of the topological monoid $[0,1]$
on $\mathfrak{M} = \nabla_2$ defined in Section \ref{sect:prelimmeasure}.

We first consider the case $n = 1$ and set, for $u \in [0,1]$ and $f \in \nabla_1$,
$$
\begin{array}{lcll}
\hat{H}_1(u,f) & =& 0 & \mbox{over } \frac{\alpha_0}{u + (1-u) \alpha_0} \mathbf{h}(f^{-1}(\{ 0 \}),u) \\
& =& 1 & \mbox{ over  its complement.}
\end{array} 
$$
and set $H_1(u,f) = \hat{H}_1(\frac{u}{\alpha_1},f)$ for $ u \leq \alpha_1$,
$H_1(u,f) = \sigma_1 \circ p_1(f) = \hat{H}_1(1,f)$ for $u > \alpha_1$,
where $\alpha_k = \la(f^{-1}(\{ k \}))$. Then $H_1 : [0,1] \times \nabla_1 \to \nabla_1$
provides an homotopy between $f \mapsto H_1(0,f) = \hat{H}_1(0,f) = f $
and $f \mapsto H_1(1,f) = \hat{H}_1(1,f) = \sigma_1 \circ p_1(f)$.
Moreover, $p_1(H_1(u,f))= p_1(f)$
since
$$
\la\left( \frac{\alpha_0}{u + (1-u) \alpha_0} \mathbf{h}(f^{-1}(\{ 0 \}),u) \right)
=
\frac{\alpha_0}{u + (1-u) \alpha_0} \left( u + (1-u)  \la(f^{-1}(\{ 0 \})) \right) = \alpha_0
$$
Finally, since $\mathbf{h}(\Omega_a,u) = \Omega_{u+(1-u)a}$ we have $\hat{H}_1(u,\sigma_1(\alpha)) = \sigma_1(\alpha)$
hence $H_1(u,\sigma_1(\alpha)) = \sigma_1(\alpha)$ for all $\alpha \in \Delta_1$ and $u \leq \alpha_1$.

\subsection{Homotopy equivalence : definition in the case of higher $n$}
\label{sect:homotophighern}
We then construct an homotopy $H_n : [0,1] \times \nabla_n \to \nabla_n$ between $\Id_{\nabla_n}$ and
$\sigma_n \circ p_n$ by induction on $n$. We use the map $\Phi$ of Section \ref{sect:prelimmeasure}.

Let $f \in \nabla_n$ with $\underline{\alpha} = p_n(f) \in \Delta_n$. We denote
$\underline{X}=(X_0\subset X_1\subset \dots)$ defined as $X_k = f^{-1}(\{ 0,\dots, k\})$
and $\underline{x} = (x_0 \leq x_1\leq )$ defined as $x_k = \alpha_0+\dots+\alpha_k$.
We set
$$
E_u^n = \frac{x_{n-1}}{u+(1-u)x_{n-1}} \mathbf{h}(X_{n-1},u), \ \ \Omega^n_u = \Omega \setminus E_u^n
$$
and, by descending induction, for $2 \leq k$ and $u \in [0,1]$,
$$
E_u^{k-1} = \Phi\left(
\frac{x_{k-2}}{x_{k-1}},E_{\bullet}^k,X_{k-2}\right)(u), \ \ 
\Omega^{k-1}_u = E_u^k \setminus E_u^{k-1}
$$
and $\Omega_{\bullet}^0 = E_{\bullet}^1$.

\begin{figure}
\begin{center}
\begin{tikzpicture}
\draw (0,0) -- (4,0) -- (4,-4) -- (0,-4) -- cycle;
\fill[blue] (3,-1) -- (4,-1) -- (4,-4) -- (3,-4) -- cycle;
\fill[yellow] (2,-2) -- (3,-2) -- (3,-4) -- (2,-4) -- cycle;
\fill[shading=axis,rectangle, top color = gray!30, bottom color = orange] (0,0) -- (3,0) -- (3,-1) -- (0,-1) -- cycle;
\draw[ultra thick, red] (0,-4) -- (2,-4);
\fill[shading=axis,rectangle, top color = gray!30, bottom color = blue] (0,0) -- (4,0) -- (4,-1) -- (3,-1) -- cycle;
\fill[red] (0,-2) -- (2,-2) -- (2,-4) -- (0,-4) -- cycle;
\fill[shading=axis, top color = orange, bottom color = red] (0,-1) -- (3,-1) -- (2,-2) -- (0,-2) -- cycle;

\fill[shading=axis, top color = orange, bottom color = yellow] (0,-1) -- (3,-1) -- (3,-2) -- (2,-2) -- cycle;
\draw (-0.2,0) node {$0$};
\draw (-0.2,-4) node {$1$};
\draw (2,0.5) node {$f$};
\draw (2,-4.5) node {$\sigma_n\circ p_n(f)$};
\draw (4.2,0) node {$0$};
\draw (4.2,-4) node {$1$};

\end{tikzpicture}
\end{center}
\caption{Homotopy between $\sigma_n \circ p_n$ and $\Id_{\nabla_n}$}
\label{fig:homotopsigmapsi}
\end{figure}
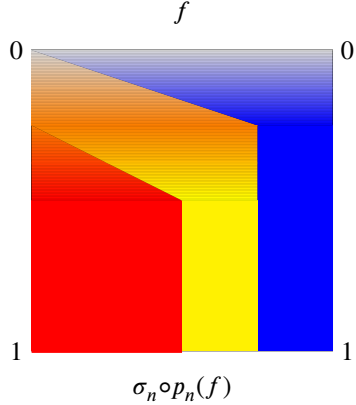

From the defining properties of the maps $\Phi$ and $\mathbf{h}$ we immediately get
\begin{itemize}
\item For all $u$, $\la(E_u^n) = x_{n-1}$ and $\la(\Omega_u^n) = \alpha_n$
\item By induction on $k$, for all $k \geq 1$, $\la(E_{\bullet}^k) = x_{k-1}$
\item hence , for all $k \geq 0$, $\la(\Omega_{\bullet}^k) = \alpha_k$, with $\Omega_0^k = f^{-1}(\{ k \})$.
\item For all $r < k$, $\Omega_u^r \subset E_u^{r+1} \subset \dots \subset E_u^k$
hence $\Omega_u^r \cap \Omega_u^k = \emptyset$ and, when $u$ is fixed, the
$\Omega_u^k$ form a partition of $\Omega$, as $\sum_k \la(\Omega_u^k) = 1$.
\item $\Omega_1^n = ]x_{n-1},1] = ]1-\alpha_n,1]$ 
\end{itemize}
hence we can define $\check{H}_n(u,f)(t) = k$ if $t \in \Omega_u^k$,
and $H_n(u,f) = \check{H}_n(\frac{u}{\alpha_n},f)$ for $0 \leq u \leq \alpha_n$, and, for $\alpha_n \leq u \leq 1$,
$$
\begin{array}{lcll}
H_n(u,f)(t) &=& H_{n-1}(\frac{u-\alpha_n}{1-\alpha_n},\check{H}_n(1,f)_{[0,1-\alpha_n]})(\frac{t}{1-\alpha_n}) 
& \mbox{ for } 0 \leq t \leq 1- \alpha_n \\
& = & n & \mbox{ for }   1- \alpha_n \leq t \leq 1 \\
\end{array}$$
 where $\check{H}_n(1,f)_{[0,1-\alpha_n]} \in \nabla_{n-1}$ is defined by
$$
u \mapsto \check{H}_n(1,f)\left(
(1-\alpha_n) u
\right). 
$$
 In the case $n = 2$, this homotopy
is depicted in Figure \ref{fig:homotopsigmapsi}.

\medskip

We check that $H_n(0,f)= \check{H}_n(0,f) = f$ since $\Omega_0^k = f^{-1}(\{ k \})$,
and, using the induction assumption, that, for $t \leq 1- \alpha_n$, we have  
$$
H_n(1,f)(t)=H_{n-1}(1,\check{H}_n(1,f)_{[0,1-\alpha_n]})(\frac{t}{1-\alpha_n})=
\sigma_{n-1} \circ p_{n-1} (\check{H}_n(1,f)_{[0,1-\alpha_n]})(\frac{t}{1-\alpha_n})
$$
and, since $\la(\Omega_{\bullet}^k) = \alpha_k$, we know that, for $k < n$,
$p_{n-1} (\check{H}_n(1,f)_{[0,1-\alpha_n]})_k = \frac{\alpha_k}{1-\alpha_n}$
hence   $$\sigma_{n-1} \circ p_{n-1} (\check{H}_n(1,f)_{[0,1-\alpha_n]}) (\frac{t}{1-\alpha_n})
 = \sigma_n \circ p_n(f)(t).$$
  Moreover, if $t \geq 1- \alpha_n$, then $H_n(1,f)(t) = n = \sigma_n \circ p_n(f)(t)$, whence $H_n(1,f) = \sigma_n \circ p_n (f)$.

\bigskip

\subsection{Proof of Proposition \ref{prop:commuteface} : generalities}

We check by induction on $n$ that the maps $H_n$ satisfy the properties of Proposition \ref{prop:commuteface},
and first of all that it indeed provides an homotopy from $\Id_{\nabla_n}$
to $\sigma_n \circ p_n$.

For $f \in \nabla_n$, we have $H_n(0,f)= \check{H}_n(0,f) = f$, since $\Omega_0^k = f^{-1}(\{ k \})$. Moreover, using the induction assumption, we have that, for $t \leq 1- \alpha_n$,   
$$
H_n(1,f)(t)=H_{n-1}(1,\check{H}_n(1,f)_{[0,1-\alpha_n]})(\frac{t}{1-\alpha_n})=
\sigma_{n-1} \circ p_{n-1} (\check{H}_n(1,f)_{[0,1-\alpha_n]})(\frac{t}{1-\alpha_n})
$$
and, since $\la(\Omega_{\bullet}^k) = \alpha_k$, we know that, for $k < n$,
$p_{n-1} (\check{H}_n(1,f)_{[0,1-\alpha_n]})_k = \frac{\alpha_k}{1-\alpha_n}$
hence   
$$\sigma_{n-1} \circ p_{n-1} (\check{H}_n(1,f)_{[0,1-\alpha_n]}) (\frac{t}{1-\alpha_n})
 = \sigma_n \circ p_n(f)(t).$$
Finally, if $t \geq 1- \alpha_n$, then $H_n(1,f)(t) = n = \sigma_n \circ p_n(f)(t)$, whence $H_n(1,f) = \sigma_n \circ p_n (f)$ and $H_n$ is indeed an homotopy from $\Id_{\nabla_n}$
to $\sigma_n \circ p_n$.

\bigskip

We now prove that $p_n( H_n(u,f)) = p_n(f)$ for all $u \in [0,1]$.
For $u \leq \alpha_n$ this is clear as $p_n( H_n(u,f)) = p_n( \check{H}_n(\frac{u}{\alpha_n},f))
= p_n(f)$ since $\la(\Omega_{\bullet}^k) = \alpha_k$ for all $k \geq 1$.
For $u \geq \alpha_n$ we have $p_n(H_n(u,f))_n = \alpha_n = p_n(f)_n$ and, for $k < n$,
$$
p_n(H_n(u,f))_k
= (1-\alpha_n)p_{n-1}(H_{n-1}(\frac{u-\alpha_n}{1-\alpha_n},\check{H}_n(1,f)_{[0,1-\alpha_n]}))_k
= (1-\alpha_n) p_{n-1}(\check{H}_n(1,f)_{[0,1-\alpha_n]})_k = \alpha_k
$$
hence $p_n(H_n(u,f)) = p_n(f)$ and this proves the claim.

Let us now consider the case $f = \sigma_n(\alpha)$,
and prove $H_n(u,f) = f$ for every $u \in [0,1]$. 
From the property $\mathbf{h}(\Omega_a,u) = \Omega_{u+(1-u)a}$ we get $E_{\bullet}^n =
\Omega_{x_{n-1}}$. From the explicit construction of $\Phi$ in Section \ref{sect:prelimmeasure} 
we get that, for $0 \leq a \leq b \leq 1$,

$$
\Phi\left( \frac{a}{b},\Omega_b,\Omega_a \right)(u) = 
\check{\mathbf{g}}\left( \Omega_a \cap \Omega_b, \frac{a}{\la(\Omega_a \cap \Omega_b)}\right)
=
\check{\mathbf{g}}\left( \Omega_a , \frac{a}{a}\right)
= \Omega_a$$
It follows by descending induction on $k$ that $E_{\bullet}^k = \Omega_{x_{k-1}}$
whence $\check{H}_n(u,\sigma_n(\alpha)) = \sigma_n(\alpha)$. Clearly then $H_n(u,\sigma_n(\alpha)) = \sigma_n(\alpha)$ for $u \leq \alpha_n$. 
Letting $\beta_k = \alpha_k/(1-\alpha_n)$ for $k < n$, we have $\sigma_n(\alpha)_{[0,1-\alpha_n]}
= \sigma_{n-1}(\beta)$ hence, for $u \geq \alpha_n$,
$$
H_n(u,\sigma_n(\alpha))(t) = H_{n-1}(\frac{u-\alpha_n}{1-\alpha_n},\sigma_{n-1}(\beta))(\frac{t}{1-\alpha_n}) 
= \sigma_{n-1}(\beta)(\frac{t}{1-\alpha_n}) = \sigma_n(\alpha)(t)
$$
for $t \leq 1-\alpha_n$, and $H_n(u,\sigma_n(\alpha))(t) = n = \sigma_n(\alpha)(t)$ for $t > 1- \alpha_n$,
and we get $H_n(u,\sigma_n(\alpha)) = \sigma_n(\alpha)$ for all $u,\alpha$.

It remains to prove that $H_n$ commutes with the face maps.

\subsection{Proof of Proposition \ref{prop:commuteface} : face maps}

We consider $D^i_{RV} : \nabla_n \to \nabla_{n+1}$ for $0 \leq i \leq n+1$
and denote $\hat{E}_u^k$, $\hat{\Omega}_u^k$ the sets $E_u^k$, $\Omega_u^k$ for
$D^i_{RV}(f)$. We set $B_k = f^{-1}([k])$, $\hat{B} = (D^i_{RV}f)^{-1}([k])$, 
$\alpha_k = \la(f^{-1}(\{ k \}))$,  $\check{\alpha}_k = \la((D^i_{RV}f)^{-1}(\{ k \}))$,
$A_k = \la(B_k)$, $\hat{A}_k = \la(\hat{B}_k)$. We have 
$A_k = \alpha_0 + \dots + \alpha_k$, $\hat{A}_k = \hat{\alpha}_0 + \dots + \hat{\alpha}_k$.

We have $\hat{\alpha}_k = \alpha_k$ for $k \leq i-1$, $\hat{\alpha}_i = 0$, and $\hat{\alpha}_{k+1} = \alpha_k$
for $k \geq i$. Also, $\hat{B}_0 = B_0$, \dots, $\hat{B}_{i-1} = B_{i-1}$
and $\hat{B}_i = \hat{B}_{i-1} = B_{i-1}$, $\hat{B}_{i+1} = B_i$, \dots, $\hat{B}_{n+1} = B_n$.
As a consequence, $\hat{A}_0 = A_0$, \dots, $\hat{A}_{i-1} = A_{i-1}$
and $\hat{A}_i = \hat{A}_{i-1} = A_{i-1}$, $\hat{A}_{i+1} = A_i$, \dots, $\hat{A}_{n+1} = A_n$.

We want to prove that
$$
\forall n\, \forall i \in [0,n+1] \, \forall u \in [0,1]\, H_{n+1}(u,D^{i}_{RV} f)= D^{i}_{RV} H_n(u,f)
$$
We prove this by induction on $n$.
\medskip

For a given $n$ we then prove this by descending induction on $i$.
Therefore we start assuming $i = n+1$. Then $D^i_{RV}$ is the inclusion
$\nabla_n \subset \nabla_{n+1}$, and we have $\check{\alpha}_{n+1} = 0$, hence $1-\check{\alpha}_{n+1} = 1$ and
$$
H_{n+1}(u,D^{n+1}_{RV} f)= \check{H}_n(u,D^{n+1}_{RV} f) = H_n(u,f) = D^{n+1}_{RV} H_n(u,f)
$$
by definition.
 
We then assume that $i < n+1$. Then 
$$
\hat{E}_u^{n+1} = \frac{\hat{A}_n}{u+(1-u)\hat{A}_n} \mathbf{h}(\hat{B}_n,u)
=\frac{A_{n-1}}{u+(1-u)A_{n-1}} \mathbf{h}(B_{n-1},u)=E_u^n$$

hence $\hat{E}_{\bullet}^{n+1} = E_{\bullet}^n$
and $\hat{\Omega}_{\bullet}^{n+1} = \Omega_{\bullet}^n$. Then,
by descending induction on $k$, we have that, provided $k > i$,
$$
\hat{E}_{\bullet}^k = \Phi\left( \frac{\hat{A}_{k-1}}{\hat{A}_k}, \hat{E}_{\bullet}^{k+1},
\hat{B}_{k-1})\right)
= \Phi\left( \frac{A_{k-2}}{A_{k-1}}, E_{\bullet}^{k},
B_{k-2}\right) = E_{\bullet}^{k-1}
$$

Therefore, $\hat{E}_{\bullet}^{k+1} = E_{\bullet}^{k}(f)$ when $k \geq i$.

Then, 
$$
\hat{E}_{\bullet}^i 
= \Phi\left( \frac{\hat{A}_{i-1}}{\hat{A}_i},\hat{E}_{\bullet}^{i+1},\hat{B}_{i-1} \right)
= \Phi\left( \frac{A_{i-1}}{A_{i-1}},E_{\bullet}^{i},B_{i-1} \right)
= \Phi\left( 1,E_{\bullet}^{i},B_{i-1} \right) = E_{\bullet}^{i}
$$
since one always has $\Phi(1,E_{\bullet},A) = E_{\bullet}$.

Then, for $k < i$, we
prove that
$$
\hat{E}_{\bullet}^k 
= \Phi\left( \frac{\hat{A}_{k-1}}{\hat{A}_k},\hat{E}_{\bullet}^{k+1},\hat{B}_{k-1} \right)
= \Phi\left( \frac{A_{k-1}}{A_{k}},E_{\bullet}^{k+1},B_{k-1} \right) = E_{\bullet}^k
$$
again by descending induction on $k$. Summarizing, one gets 
$\hat{E}_{\bullet}^{k+1} = E_{\bullet}^{k}$ for $k \geq i$,
$\hat{E}_{\bullet}^{k} = E_{\bullet}^{k}$ for $k \leq i$. It follows that
$\hat{\Omega}_{\bullet}^{k} = E_{\bullet}^{k} \setminus E_{\bullet}^{k-1} = \Omega_{\bullet}^{k-1}$
for $k > i$, $\hat{\Omega}_{\bullet}^i = \emptyset$,
$\hat{\Omega}_{\bullet}^{k} = \Omega_{\bullet}^{k}$ for $k < i$. Thus, from the
definition we get $\check{H}_{n+1}(u,D^i_{RV}f) = D_i^{RV} \check{H}_n(u,f)$
for all $u \in [0,1]$ hence
and
$$H_{n+1}(u,D^i_{RV}f)
= \check{H}_{n+1}(\frac{u}{\check{\alpha}_{n+1}},D^i_{RV}f)
= \check{H}_{n+1}(\frac{u}{\check{\alpha}_{n}},D^i_{RV}f)
= D^i_{RV}\check{H}_n(\frac{u}{\check{\alpha}_{n}},f)
= D^i_{RV} H_n(u,f)
$$
for $0 \leq u \leq \check{\alpha}_{n+1} = \alpha_n$.
Then, for $\check{\alpha}_{n+1} = \alpha_n \leq u \leq 1$, we have over $[0,1-\check{\alpha}_{n+1}]$ that
$$H_{n+1}(u,D^i_{RV}f)(t) = H_n\left(
\frac{u-\check{\alpha}_{n+1}}{1-\check{\alpha}_{n+1}},
\check{H}_{n+1}(1,(D^i_{RV}f)_{|[0,1-\check{\alpha}_{n+1}]}) \right)\left(\frac{t}{1-\check{\alpha}_{n+1}}\right) 
$${}$$
=
H_n\left(
\frac{u-\alpha_{n}}{1-\alpha_n},
\check{H}_{n+1}(1,D^i_{RV}(f_{|[0,1-\alpha_{n}]}))
\right)\left(\frac{t}{1-\check{\alpha}_{n+1}}\right)
$$
{}
$$=
H_n\left(
\frac{u-\alpha_{n}}{1-\alpha_n},
D^i_{RV}\check{H}_{n}(1,f_{|[0,1-\alpha_{n}]})
\right)\left(\frac{t}{1-\check{\alpha}_{n+1}}\right)
$${}$$=
D^i_{RV} H_{n-1}\left(
\frac{u-\alpha_{n}}{1-\alpha_n},
\check{H}_{n}(1,f_{|[0,1-\alpha_{n}]}))
\right) \left(\frac{t}{1-\alpha_n}\right)
=
D^i_{RV} H_n(u,f) (t)
$$
and for $t \geq 1 - \check{\alpha}_{n+1}= 1 - \alpha_n$ we have
$H_{n+1}(u,D^i_{RV}f)(t) = n+1$ while $ D^i_{RV} H_n(u,f) (t) = D_i^c(n) = n+1$
whence $H_{n+1}(u,D^i_{RV}f) = D^i_{RV} H_n(u,f)$ 
and this proves the claim.

\begin{remark} It can be checked by explicit computations that the above
construction is \emph{not} compatible with the degeneracy maps,
already for $f \in \nabla_2$ being
$f([0,1/3[) = 2$,$f([1/3,2/3]) = 1$ and $f(]2/3,1]) = 0$,
with $\Omega_t = [0,t]$ and $[0,1]$ is endowed with the Lebesgue measure.
\end{remark}

\section{Additional properties}

\subsection{Kan condition for $\Sing_{RV} X$}
\label{sect:kan}
Here we prove that, for $X$ a topological space, the simplicial set $\Sing_{RV} X$ satisfies
the Kan condition. For this,
consider some $n \geq 1$, $k \in [n]$, $z_i \in (\Sing_{RV} X)_{n-1} = (\Sing_{RV} X)([n-1])$
for $i \in [n]$ with $i \neq k$, such that $z_i.D_j^c = z_j. D_{i-1}^c$ for $0 \leq j < i \leq n$ and $k \not\in \{ i,j \}$. In order to prove that $\Sing_{RV} X$ satisfies the Kan condition, we need (see e.g. \cite{FRPIC} \S 4.5) to find $z \in (\Sing_{RV} X)_n$ such that $z_i = z.D_i^c$ for $i \neq k$.

For $\underline{a} = (a_0,\dots,a_n) \in \Delta_n$, denote $m_k(\underline{a}) = \min_{r \neq k} a_r$,
and define $q(\underline{a}) \in \Delta_n$ by $q(\underline{a})_i = a_i - m_k(\underline{a})$ for $i \neq k$,
$q(\underline{a})_k = a_k + n m_k(\underline{a})$. This defines a continuous map which is actually a
retraction from $\Delta_n$
onto its $k$-th horn $\Lambda_n^k = \{ (a_0,\dots,a_n) \ | \ \exists i \neq k \ \ a_i = 0 \}$.
We want to define $z_n : \nabla_n \to X$. Let us represent elements of $\nabla_n$ by $(n+1)$-tuples $\underline{A}= (A_0,\dots,A_n)$
providing a partition of $\Omega$ into measurable sets, with $A_i$ the preimage of $i \in [n]$ by the
corresponding random variable $\Omega \to [n]$ and introduce
$$
V_n^k = p_n^{-1}(\Lambda_n^ k) = \{ \underline{A}= (A_0,\dots,A_n) \in \nabla_n \ | \ \exists i \neq k \la(A_i) = 0 \}
$$
The $z_i \in (\Sing_{RV} X)_{n-1}$ for $i \neq k$ can immediately be glued together
into a map $\check{z} : V_n^k \to X$. 

Then, consider the map $\nabla_n \to \nabla_n$
defined by mapping  $\underline{A}= (A_0,\dots,A_n)$ to $\underline{B} = (B_0,\dots,B_n)$
with $B_k$ equal to the complement in $\Omega$ of $\bigcup_{r \neq k} B_r$
and, for $i \neq k$, $B_i = A_i$ if $\la(A_i) = 0$, and  $B_i = \mathbf{g}(A_i,\min_{r \neq k} \la(A_r)/\la(A_i))$
otherwise, where $\mathbf{g}$ is as in Section \ref{sect:homotopn1}. We have $B_i \subset A_i$ for $i \neq k$, hence the $B_i$'s still form a partition of $\Omega$,
and it is easily checked that $\la(B_i) = q(\underline{a})_i$ for $\underline{a} = p_n(\underline{A}) \in \Delta_n$.
In particular $\underline{B} \in V_n^k$ and this defines a continuous retraction $\nabla_n \to V_n^k$. Composing
it with $\check{z} : V_n^k \to X$ provides $z \in \Sing_{RV} X$ such that $z_i = z.D_i^c$ for $i \neq k$, and this
proves the claim.

Although it is not needed for the proof, 
as in the classical case one can prove that
$V_n^k$ is a \emph{deformation} retract of $\nabla_n$, see Proposition \ref{prop:VnrNablanstrongdefretract} below.

\subsection{Ordered simplicial complexes}
\label{sect:orderedSC}
Let $\KK$ be a simplicial complex over a \emph{totally ordered} set $S$ of vertices. Then
a simplicial set $S\KK$ can be obtained in a standard way by repeating
vertices, namely 
$$S\KK_n = \{ (s_0,\dots,s_n) \in S^{n+1} ; s_0 \leq s_1 \leq \dots \leq s_n \ \& \ \{ s_0,\dots,s_n \} \in \KK \}
$$
and, for $f \in \Hom_{\Delta}( [m], [n])$,  $S\KK(f)$ maps $(s_0,\dots,s_n)$
to $(s_{f(0)},\dots,s_{f(m)})$. A non-degenerate $n$-simplex is characterized by the property $s_0 < \dots < s_n$.

In \cite{SRV} we defined a (metric) space of simplicial random variables $L(\KK)$
associated to the simplicial complex $\KK$. Letting $|\KK|_w$ and $|\KK|_m$ be the geometric realizations of $\KK$ equipped with the weak and strong (or metric) topology, respectively,
the proof of Theorem \ref{theo:comprealKK} consists in the following chain of homotopy
equivalences
$$
L(S\KK) \sim |S \KK| \sim |\KK|_w \sim |\KK|_m \sim L(\KK)
$$
where the first one is given by Theorem \ref{theo:LwFw}, the second one is standard (see e.g. \cite{GELFANDMANIN} I.2.13), the third one is Dowker's theorem (\cite{DOWKER}; see also \cite{CWMILNOR}), and the fourth one
is our Theorem 1 of \cite{SRV}.

Let us consider $M \KK$ the associated pre-simplicial set, namely the restriction of $S \KK$
to the category $\mathbf{M}$. As expected, the topological spaces $L(M \KK)$ and $L(S \KK)$
are actually the same.

\begin{proposition} \label{prop:realhomeo}
The spaces $L( M\KK) $ and $L(S \KK)$ are homeomorphic.
\end{proposition}
\begin{proof}

Let us denote $F = S \KK$. Then $|F|$ is a quotient space of $\bigsqcup F_n^{\#} \times \Delta_n$,
equal to the image of $\bigsqcup F_n \times \Delta_n$ modulo the equivalence
relations $(\alpha f, a) \sim (\alpha ,\Delta(f)(a))$ for $f \in \DDelta$. Now let us assume that $(\alpha f,a) \sim (\alpha,\Delta(f)(a))$
for some $\alpha \in F_n^{\#}$, $\alpha f \in F_m^{\#}$ and $f \in \Hom_{\DDelta}([m],[n])$. We have
$\alpha = (s_0,\dots,s_n)$ and $\alpha f = (s_{f(0)},\dots,s_{f(m)})$. But $\alpha f \in F_m^{\#}$
means $s_{f(0)}<\dots<s_{f(m)}$ and this implies that $f$ is injective. Therefore $|S \KK| = \| M \KK \|
= (\bigsqcup F_n^{\#} \times \Delta_n)/\sim$ with $\sim$ generated by $(\alpha f, a) \sim (\alpha ,\Delta(f)(a))$
for $f \in \mathbf{M}$.

 The same argument shows that $L( M\KK) $ and $L(S \KK)$ are homeomorphic.

The (weak) geometric realization of $\KK$ is the topological union
of the $|\KK^{(n)}|$ where $\KK^{(n)}$ is the $n$-skeleton of $\KK$, defined as the
collection of elements of $\KK$ of cardinality at most $n+1$. From the cellular structure
of $\KK^{(n)}$ and since $\mathbf{M}$ is generated by the face maps we get
immediately that the identity map of $\bigsqcup_{k \leq n} F_k \times \Delta_k$ induces
an homeomorphism between $|\KK^{(n)}|$ and $\| (M \KK)^{(n)} \|$. From this we get a commutative
ladder of homeomorphisms which induces an homeomorphism between $|\KK|$ and $\| M \KK  \|$.

\end{proof}

\subsection{Finite products and equalizers}

Let $F,G$ be two simplicial sets, and $F \times G$ their product in the category $\sSet$. It is defined
(see e.g. \cite{FRPIC})
by $(F \times G)_n = F_n \times G_n$, and $(\alpha,\beta).\sigma = (\alpha.\sigma,\beta.\sigma)$
for $\sigma \in \DDelta$. From the product property we have natural maps $|F \times G| \to |F|\times |G|$
and $L(F \times G) \to L(F)\times L(G)$. The former is known to be an homeomorphism (\cite{FRPIC}
Proposition 4.3.15), and more generally the geometric realization functor preserves finite limits.
Here we show that the latter map $L(F \times G) \to L(F)\times L(G)$
is \emph{not} an homeomorphism in general.

Our example is the following one. We consider the simplicial complex $\PF^*([1])$ on the vertex set $S = [1] =  \{ 0, 1 \}$
given by the collection of all non-empty subsets of $S$. The simplicial set $F$
associated to it is the $1$-simplex, considered as a simplicial set, and $|F| = \Delta_1 = I$,
$L(F) = \nabla_1 = \mathfrak{M}$.

The simplicial set $F \times F$ can be described as $F_{\KK}$ for $\KK$ the
2-dimensional simplicial complex on the set $(F \times F)_0 = S \times S$ whose maximal simplices
are $\{ (0,0), (0,1), (1,1) \}$ and $\{ (0,0), (1,1), (1,0) \}$. An element of
$L(F \times F)$ can thus be described as an element of $L^1(\Omega,S\times S)$, or
equivalently by a partition of $\Omega$ (up to neglectability) into 4 parts
$\underline{A} = (A_{00},A_{01},A_{10},A_{11})$. The condition that it belongs to $L(F \times F)$
reads $A_{10} = \emptyset$ or $A_{01} = \emptyset$.

We now consider its image under the projection map $L(F \times F) \to \nabla_1 \times \nabla_1 = \mathfrak{M} \times \mathfrak{M}$, where elements of $\nabla_1$ are
identified with (classes of measurable) subsets of $\Omega$. It is easily checked that $\underline{A}$
is mapped to $(A_{10}\cup A_{11},A_{01}\cup A_{11})$, which is equal either to 
$( A_{11},A_{01}\cup A_{11})$ if $A_{10} = \emptyset$, or to
$(A_{10}\cup A_{11}, A_{11})$ if $A_{01} = \emptyset$. From this one gets that the image
of $L(F \times F)$ inside $L(F) \times L(F)$ is the
collection of pairs $(U,V) \in
\mathfrak{M}\times \mathfrak{M}$
such that either $U \subset V$
or $V \subset U$. Therefore the map is not surjective,
and this proves that $L : \sSet \to \Top$ does \emph{not} preserve finite products.

A positive property however is that $L$ preserves equalizers.
\begin{proposition}
Let $F,G$ be two simplicial sets,  $\varphi, \psi : F \to G$ two simplicial maps,
and $H \subset F$ their equalizer in $\sSet$. Then $L(H)$ is the equalizer of the
maps $L(\varphi), L(\psi) : L(F) \to L(G)$.
\end{proposition}
\begin{proof}
Clearly $L(H)$ is included inside this equalizer. Conversely, let us consider
one of its elements, and a representative $(\alpha,a)$ of it with $a$ an interior point.
By definition $(f(\alpha),a)$ and $(g(\alpha),a)$ are equivalent inside $L(F)$.
Let $\beta,\beta' \in F^{\#}$ and $S,S' \in \mathbf{E}$ such
that $f(\alpha) = \beta.S$, and $g(\alpha) = \beta'.S'$. We
have $(f\alpha),a) = (\beta.S,a) \sim (\beta,\nabla(S).a)$ and
similarly $(g(\alpha),a) \sim (\beta',\nabla(S').a)$.
Now by definition $(f(\alpha),a)$ and $(g(\alpha),a)$ are equivalent inside $L(F)$,
so their minimal representatives are the same, that is $\beta = \beta'$ and $\nabla(S).a = \nabla(S').a$.
Since $a$ is interior one gets $S = S'$ by Lemma \ref{lem:propsnabla} (3) and this
proves $f(\alpha) = \beta.S = \beta'.S' = g(\alpha)$, which proves the claim.
\end{proof}

Actually $L$ also \emph{reflects} equalizers, as is immediate by application of the natural
transformation $p : L \leadsto |\bullet |$ and the similar result in the classical case
(see e.g. \cite{FRPIC} Proposition 4.3.13).

\subsection{Standard variation}

There is a well-known standard variation on the definition of a $n$-simplex,
which is better behaved for some purposes, 
as $\lltriangle_n = \{ (x_1,\dots,x_n) \in [0,1] ; x_1\leq \dots \leq x_n \}$. More precisely,
it defines a functor
$\lltriangle : \DDelta \to \Top$, mapping $\sigma \in \DDelta$ to $\lltriangle(\sigma)$ defined by
$$\lltriangle(\sigma)(x_1,\dots,x_n) = \left( x_{1+\sup \sigma^{-1}(\{ 0,\dots,i-1 \})} \right)_{i=1,\dots,m }
$$
with the conventions $\sup \emptyset = - \infty$ and $x_{-\infty} = 0$.

 There is a
(bicontinuous) bijection between $\lltriangle_n$ and $\Delta_n = \{ (a_0,\dots,a_n) \in [0,1]^n ; a_0+\dots+ a_n = 1 \}$, given by
$x_0 = a_0$, $x_1 = a_0 + a_1$, \dots, $x_n = a_0+\dots+a_{n-1}$
and conversely $a_0 = x_0$, $a_i = x_i - x_{i-1}$ for $0 < i < n$, $a_n = 1-x_n$.
It is immediately checked that this bijection defines an isomorphism between the functors $\Delta$ and $\lltriangle$.

There is a similar variation for simplicial random variables. There is a cosimplicial space
$\ultriangle : \DDelta \to \Top$ corresponding to
$$
\ultriangle_n = \{ \underline{A} = (A_1,\dots, A_n) ; A_1 \subset A_2 \subset \dots \subset A_n \subset \Omega \}
$$
where the subsets of $\Omega$ are always understood up to subsets of measure $0$,
and the topology is for instance given by the metric $d(\underline{A}^1,\underline{A}^2) = \max_i \la(A_i^1 \Delta A_i^2)$ where $U \Delta V = (U \setminus V)\cup (V\setminus U)$.  For $\sigma \in \DDelta$,
$$\ultriangle(\sigma)(A_1,\dots,A_n) = \left( A_{1+\sup \sigma^{-1}(\{ 0,\dots,i-1 \})} \right)_{i=1,\dots,m }
$$
This functor is isomorphic to $\nabla$. To see this, recall that $\nabla_n = L^1(\Omega,[n])$
can be identified with
$\{ (B_0,\dots,B_n) ; \Omega = B_0\sqcup \dots \sqcup B_n \}$ (where set-theoretic equalities are always
understood up to neglectable subsets) through $B_i = f^{-1}(\{ i \})$ for $i \in [n]$
and $f \in \nabla_n$. Then, the correspondence is given
via $A_k = B_0\cup \dots \cup B_{k-1}$, and it is easily checked to be simplicial. Finally,
the probability-law maps $p'_n : \ultriangle_n \to \lltriangle_n$, $(A_1,\dots,A_n) \mapsto (\la(A_1),\dots,\la(A_n)) \}$ are continuous,
fit together, and define a natural transformation $\ultriangle \leadsto \lltriangle$. It is easily checked
that all this fits into a commutative diagram of functors as follows.
$$
\xymatrix{
\nabla \ar@{<->}[r] \ar[d] &\ultriangle \ar[d] \\
\Delta \ar@{<->}[r] & \lltriangle
}
$$

\section{Homotopy invariance}

In this section we investigate the relation between the functors $\Sing$ and $\Sing_{RV}$.
Recall that $p_n : \nabla_n \to \Delta_n$ denotes the probability-law map. For $Z$ a topological space,
$\Sing_{RV} Z$ is defined as the simplicial set with $n$-vertices the (continuous) maps $\nabla_n \to Z$,
and $\sigma \in \DDelta$ acts on $f \in (\Sing_{RV} Z)_n$ as
$f.\sigma = f \circ \nabla(\sigma)$. If $(Z,z_0)$ is a pointed space, $\Sing_{RV} Z$ is a pointed simplicial set,
with distinguished vertex the map $\nabla_0 \to Z$ mapping (the single element of) $\nabla_0$ to $z_0$, and $\Sing_{RV}$
clearly defines a functor $\Top_* \to \sSet_*$, right adjoint to $L : \sSet_* \to \Top_*$.
Also recall that the monomorphisms of $\sSet$ are the injective simplicial maps between the
underlying graded sets, so we call them injective morphisms.

\begin{proposition}
For every (pointed) topological space $Z$, the maps $f \mapsto f\circ p_n$, $(\Sing Z)_n \to (\Sing_{RV} Z)_n$ induce
an injective morphism of (pointed) simplicial sets $R_Z : \Sing Z \to \Sing_{RV} Z$. The collection $(R_Z)_Z$ defines a natural transformation
$\Sing \leadsto \Sing_{RV}$.
\end{proposition}
\begin{proof}
Let $f \in (\Sing Z)_n$ and $\sigma \in \Hom_{\DDelta}([m],[n]) = \Hom_{\DDelta^{op}}([n],[m])$.
We want to prove $R_Z(f.\sigma) = R_Z(f).\sigma$. We have
$R_Z(f.\sigma) = R_Z(f \circ \Delta(\sigma)) = f \circ \Delta(\sigma) \circ p_m$
and $R_Z(f).\sigma = R_Z(f) \circ \nabla(\sigma) = f \circ p_n \circ \nabla(\sigma)$.
In the proof of Proposition \ref{prop:ptransfonat} we have got that $p_n \circ \nabla(\sigma) = \Delta(\sigma) \circ p_m$
hence $R_Z(f.\sigma) = R_Z(f).\sigma$ and we have a morphism of simplicial sets.
Its injectivity is an immediate consequence of the surjectivity of the maps $p_n : \nabla_n \onto \Delta_n$. Finally,
if $Z$ is pointed, it
clearly maps base point to base point.

Let $\varphi : Z \to T$ a continuous map. We now want to prove that $R_T \circ \Sing(\varphi) = \Sing_{RV}(\varphi) \circ R_Z$. Let $f \in (\Sing Z)_n$. We have $\Sing(\varphi)(f) = \varphi \circ f$
and $R_T(\Sing(\varphi)(f)) = (\varphi \circ f) \circ p_n$
and $\Sing_{RV}(\varphi)(R_Z(f)) = \Sing_{RV}(\varphi)(f\circ p_n)= \varphi \circ (f \circ p_n)$
and this proves the claim.
\end{proof}

For $\KK$ and $\mathcal{L}$ two ordered simplicial
complexes, we have $F_{\KK} \times F_{\mathcal{L}} = F_{\mathcal{K}\times \mathcal{L}}$,
where $\mathcal{K}\times \mathcal{L}$ is the simplicial complex with vertices $(x,y)$ for $x \in \bigcup \mathcal{K}$
and $y \in \bigcup \mathcal{L}$ and simplices the $\{ (x_0,y_0),(x_1,y_1),\dots,(x_n,y_n) \}$
such that $x_0 \leq x_1 \leq \dots \leq x_n$, $y_0 \leq y_1 \leq \dots \leq y_n$. The total ordering chosen on the
vertices is for instance the lexicographic ordering -- or any other total ordering refining the diagonal partial ordering
$(a,b) \leq (a',b')$ iff $a \leq b$ and $a' \leq b'$.

In particular, let $\mathcal{I} = \PF^*([1])$ the simplicial complex whose geometric realization is $\Delta_1$.
Let $S$ be the (ordered) vertex set of $\KK$.
For $x \in S$ we set $x^- = (x,0)$, $x^+ = (x,1)$.
The simplices of $\mathcal{K} \times \mathcal{I}$ are the 
$$\{ x_0^-,x_1^-,\dots,x_{r-1}^-,x_{r}^+,x_{r+1}^+,\dots, x_{m}^+ \}$$
such that $x_0 \leq \dots \leq x_{m}$
and $\{ x_0,\dots,x_{m} \} \in \KK$.

We have natural maps $L(\KK)  \times \{ 0, 1 \} \to L(\mathcal{K} \times \mathcal{I})$ and  $L(\mathcal{K} \times \mathcal{I}) \to L(\mathcal{K})$ providing a factorization
of the first projection map $L(\KK)  \times \{ 0, 1 \} \to L(\KK)$. In the
forthcoming sections we are going to prove the following
property.
\begin{proposition} \label{prop:goodcylinder}
Let $\KK$ be a finite simplicial complex. Then $L(\KK \times \mathcal{I})$ is a cylinder
object for $L(\KK)$ inside Str\o m's (closed) model category structure.
\end{proposition}
Recall from \cite{STROM} and e.g. \cite{GOERSSJARDINE} that this means that
\begin{enumerate}
\item The natural projection map $\pi : L(\KK \times \mathcal{I}) \to L(\mathcal{K})$ is
an homotopy equivalence.
\item The natural inclusion map $L(\KK) \times \{ 0, 1 \} \to L(\KK \times \mathcal{I})$ is a closed
cofibration.
\end{enumerate}
Thus, this is a \emph{good cylinder} in the terminology of \cite{DWYER}. 

Assume that $S$ is well-ordered, denote $x_0$ the minimal vertex of $\KK$. We considered $\KK$ (or $F_{\KK}$) as based
at $x_0$. Consider $\{ 0 \}$ as a base vertex for $\mathcal{I}$. The coproduct $F_{\KK} \vee F_{\KK}$
of $F_{\KK}$ with itself inside $\sSet_*$ has the form $F_{\KK \vee \KK}$
with $\KK \vee \KK$ a simplicial complex on the vertex set obtained by dividing $S \times \{ 0,1 \}$ by the relation $x_0^+ = x_0^-$.
On this vertex set, $\KK \vee \KK$ has for simplices the 
$\{  x_{i_0}^-,x_{i_1}^-,\dots,x_{i_r}^- \}$
and the 
$\{  x_{i_0}^+,x_{i_1}^+,\dots,x_{i_r}^+ \}$
 for $x_{i_0} \leq \dots \leq x_{i_r}$ and $\{ x_{i_1},\dots,x_{i_r} \} \in \KK$.

Recall that the smash product $X \wedge Y$ is in $\Top_*$ the quotient of $X \times Y$ by the
image of $X \vee Y$. It is also the push-out

$$
\xymatrix{
X \vee Y \ar[r]\ar[d] & X \times Y\ar[d] \\
1 \ar[r] & X \wedge Y
}
$$
and this definition also applies in $\sSet_*$. Denote $\mathcal{I}^+$ the simplicial complex $\mathcal{I}$ together
with an isolated vertex, taken as basepoint (for instance $\mathcal{I}^+ = \{ \{ 0 \}, \{1 \}, \{ 0,1 \},\{ -1 \} \}$).

We have a similar statement for pointed sets (recall from \cite{STROM} that a map
between pointed spaces is a cofibration, fibration, or homotopy equivalence if and only if the underlying
map between non-pointed spaces is one).

\begin{proposition} \label{prop:goodcylinderpointed}
Let $\KK$ be a finite simplicial complex with minimal vertex $x_0$. Then $L(F_{\KK} \wedge F_{\mathcal{I}^+})$ is a good cylinder
object for $L(\KK)$ inside Str\o m's (closed) model category structure on $\Top_*$.
\end{proposition}

From these propositions, which are together equivalent to the the propositions \ref{prop:propequivalence} and \ref{prop:propcofib} proven below,
 we deduce the following Theorem. Recall (see e.g. \cite{GOERSSJARDINE} p. 73) that homotopy equivalence is a
well-defined concept between objects of a (closed) model category which are both fibrant and cofibrant.
 The strategy used in its proof has been suggested to me by D. Chataur.

\begin{theorem} \label{theo:eqSingSingRV} Let $Z$ be a topological space. Then the natural morphism
$\Sing(Z) \to \Sing_{RV}(Z)$ is an homotopy equivalence.
In particular the induced map $|\Sing(Z)| \to |\Sing_{RV}(Z)|$ is an homotopy equivalence.
\end{theorem}

\begin{proof}
Since $\Sing(Z)$ and $\Sing_{RV}(Z)$ are both Kan complexes (see Section \ref{sect:kan}),
they are both fibrant and cofibrant in the standard model structure on $\sSet$. It is then
 it is enough 
 (see e.g. \cite{GOERSSJARDINE}  Theorem 1.10)
 to prove that
the induced maps $\pi_0(\Sing(Z)) \to \pi_0(\Sing_{RV}(Z))$
and $\pi_n(\Sing(Z),z_0) \to \pi_n(\Sing_{RV}(Z),z_0)$ are bijections for every $n \geq 1$ and
vertex $z_0 \in \Sing(Z)_0 = \Sing_{RV}(Z)_0 = Z$. This is what we are going to prove.

We first claim that, for $X = F_{\KK}$ the simplicial set associated to a finite simplicial complex $\KK$, then
the natural map from
$\Hom(X,\Sing(Z)) \simeq \Hom(|X|,Z)$ to $\Hom(X,\Sing_{RV}(Z)) \simeq \Hom(L(X),Z)$
becomes a bijection up to a homotopy, where up to homotopy means the
genuine homotopy relation on $\Hom(|X|,Z)$ inside the topology category,
and on $\Hom(L(X),Z)$ means the image under $L$ of the homotopy relation of $\sSet$.
In other terms, this latter equivalence relation corresponds to the cylinder $L(\KK \times \mathcal{I})$. 

We prove this claim now. Since
by Proposition \ref{prop:goodcylinder} $L(\KK \times \mathcal{I})$ is a good cylinder for Str\o m's stucture, it induces (see e.g. \cite{GOERSSJARDINE} ch. II.1 Corollary 1.9) the same homotopy equivalence as the
standard cylinder $L(\KK) \times [0,1]$, which is the genuine homotopy relation. 
Therefore
the map we want to prove it is a bijection is nothing else than the natural map $[|\KK|,Z] \to [L(\KK),Z]$ inside the naive homotopy category. Since this map is induced by
the probability-law map $L(\KK)\to|\KK|$ which is an homotopy equivalence, it is a bijection, and this
proves the claim. 

An immediate consequence of this claim, taking for $\KK$ a point, is that the induced map 
$\pi_0(\Sing(Z)) \to \pi_0(\Sing_{RV}(Z))$ is a bijection.

Now assume that the vertex set of $\KK$ is linearly ordered with a minimal element $x_0$, and that some basepoint $z_0 \in Z$ is chosen. We then make the similar
claim that, for $X = F_{\KK}$ and $\KK$ finite, 
the natural map from
$\Hom_{\sSet_*}(X,\Sing(Z)) \simeq \Hom_{\Top_*}(|X|,Z)$ to $\Hom_{\sSet_*}(X,\Sing_{RV}(Z)) \simeq \Hom_{\Top_*}(L(X),Z)$
becomes a bijection up to a homotopy, where up to homotopy means the
genuine pointed homotopy relation on $\Hom_{\Top_*}(|X|,Z)$
inside the pointed topology category,
and on $\Hom_{\Top_*}(L(X),Z)$ means the image under $L$ of the homotopy relation of $\sSet_*$.
In other terms, this equivalence relation corresponds to the cylinder $L(F_{\KK} \wedge F_{\mathcal{I}})$. This claim results from
Proposition \ref{prop:goodcylinderpointed} and the same argument as before, because
two pointed spaces are homotopically equivalent if
and only if they are freely equivalent (see \cite{STROM} and the references there).

We apply this result to $\KK = \partial \PF^*([n])$ and $X = F_{\KK}$ equal to the boundary of the $n$-simplex, for $n \geq 1$, with base point $0$.
Recall that $\pi_n(F,0) = \Hom_{\HosSet_*}(X,F)$  whenever $F$ is a Kan complex (see e.g. \cite{FRIEDMAN} \S 9).
Then
the natural map
$$\pi_n(\Sing(Z),z_0) \simeq \Hom_{\HosSet_*}(X,\Sing(Z)) \to  \Hom_{\HosSet_*}(X,\Sing_{RV}(Z)) \simeq \pi_n(\Sing_{RV}(Z),z_0)$$ is identified with $\Hom_{\HoTop_*}(|X|,Z)\to \Hom_{\HoTop_*}(L(X),Z)$. Again because the
probability-law map induces an homotopy equivalence, this map is an isomorphism and we get
that the natural map
$$
\pi_n(\Sing(Z),z_0) \to \pi_n(\Sing_{RV}(Z),z_0)$$
is an isomorphism for all $n \geq 1$. This proves the theorem.

\end{proof}

\subsection{Preliminary tools}
Notice that the topology on $L(\KK)$ is induced by the metric topology
on $L^1(\Omega,S)$ when $\KK$ is finite, and $\KK \subset \PF^*(S)$. We shall
need the following lemma only when $S$ is finite, however the extra cost
for the full statement is minimal, so we state it in full generality.

\begin{lemma} \label{lem:continuiteLOS} Let $X$ be a space, $S$ a set. For $s \in S$ we define $p_s : L^1(\Omega,S) \to \nabla_1$
given by $f \mapsto f^{-1}(s)$. Then $F : X \to L^1(\Omega,S)$ is continuous iff $\forall s \in S \ p_s \circ F : X \to \nabla_1$ is continuous.
\end{lemma}
\begin{proof}
Recall that the topology on $L^1(\Omega,S)$ is given by the metric $d(f,g) = \int_{\Omega}d(f(t),g(t)) \dd t$,
and the one on $\nabla_1 = \mathfrak{M}$ by the metric $d(U,V) = \la(U \Delta V)$. Since $\la(f^{-1}(s) \Delta g^{-1}(s))
\leq \int_{\Omega}d(f(t),g(t)) \dd t$ we get that each of the $p_s, s\in S$ is 1-Lipschitz hence
continuous. This implies that, if $F : X \to L^1(\Omega,S)$ is continuous, then so are the $p_s \circ F$, $s \in S$.

Conversely, let us choose $x_0 \in X$ and set $f_0 = F(x_0)$. We want to prove that $F$ is continuous
at $x_0$. Let $\eps > 0$. We want to prove that there exists an open neighborhood $U$ of $x_0$ such
that $x \in U \Rightarrow d(F(x),f_0) \leq \eps$.

Since $\sum_{s \in S} \la(f^{-1}(s)) = 1$, there exists $S_0 \subset S$ finite such that $\la(f^{-1}(S_0)) > 1 - \eps/2$. We set $m = \# S_0$. By continuity of $p_s \circ F$, for every $s \in S_0$ there exists
an open neighborhood $U_{s}$ of $x_0$ such that $x \in U_s$ implies that $f_x = F(x)$
satisfies $\la(f_x^{-1}(s) \Delta f_0^{-1}(s)) \leq \eps/2m$. Setting $U := \bigcap_{s \in S_0} U_s$
we get an  open neighborhood $U$ of $x_0$ such that, for every $x \in U$,
$$
d(f_0,f_x) = \sum_{s \in S} \la(f_0^{-1}(s) \setminus f_x^{-1}(s)) \leqslant\frac{\eps}{2} +
\sum_{s \in S_0} \la(f_0^{-1}(s) \setminus f_x^{-1}(s)) \leqslant \frac{\eps}{2}+ m \times \frac{\eps}{2m} = \eps
$$
and this proves the claim.
\end{proof}

Recall that $\ultriangle_n = \{ \underline{X} = (X_1,\dots,X_n) \ | \ X_1\subset X_2 \subset \dots \subset X_n \subset \Omega \}$. For $\underline{X} \in \ultriangle_n$, we set by convention $X_0 = \emptyset$ and $X_{n+1} = \Omega$. Also recall
from Section \ref{sect:prelimmeasure} the map
 $\mathbf{g} : \nabla_1 \times I \to \nabla_1$ such that $\mathbf{g}(A,0) = A$, $\la(\mathbf{g}(A,u)) = \la(A)(1-u)$ and $\mathbf{g}(A,u) \supset \mathbf{g}(A,v)$ whenever $u \leq v$.
  It moreover
satisfies that $\la(\mathbf{g}(E,u)\Delta \mathbf{g}(E,v)) \leqslant 4\la(E \Delta F) + |v-u|$.

\begin{lemma} \label{lem:interpolationJ} There exists a continuous \emph{interpolation map} $J : \ultriangle_n \times I \to \ultriangle_{1}$
such that, for $\underline{X} = (X_1,\dots,X_n) \in \ultriangle_n$, $c \in I$, and $k$ such that $\la(X_k) \leq c  < \la(X_{k+1})$, we have $X_k \subset J(\underline{X},c) \subset X_{k+1}$
and $\la(J(\underline{X},c)) = c$.
\end{lemma}
\begin{proof}
We define such a map $J$ by letting $J(\underline{X},c) = Y$ with $Y = X_k \cup \mathbf{g}(X_{k+1}\setminus X_k, a)$ with $k$ as in the statement,
and with $a \in I$ such that $\la(Y) = \la(X_k) + (\la(X_{k+1}) - \la(X_k))(1-a)=\la(X_{k+1})-a\la(X_{k+1}\setminus X_k)=c$, that is $a = (\la(X_{k+1})-c)/\la(X_{k+1}\setminus X_k)$.

It remains to prove that $J$ is a continuous map.
Let $\underline{X}^0 \in \ultriangle_n$ and $c_0 \in I$. Let $k_0$ be such that $\la(X^0_{k_0}) \leq c_0 < \la(X^0_{k_0+1})$. We separate two cases.

First assume $\la(X^0_{k_0}) < c_0$. Then, for $(\underline{X},c)$ close enough
to $(\underline{X}_0,x_0)$ we can assume that $\la(X_{k_0}) < c < \la(X_{k_0+1})$, and the
continuity of $J$ at $(\underline{X}^0,c_0)$ is an easy consequence of the continuity
of $\mathbf{g}$ and of the elementary set-theoretic operations.

We now assume $\la(X^0_{k_0}) = c_0$. Then we have some $r_0$ with $1 \leq r_0 \leq k_0$ such that
$\la(X_{r_0-1}^0)< \la(X_{r_0}^0) = \la(X_{r_0+1}^0) = \dots = \la(X_{k_0}^0) = c_0$.
Of course this implies $X_{r_0}^0 = X_{r_0+1}^0 = \dots = X_{k_0}^0$.
Then, for  $(\underline{X},c)$ close enough
to $(\underline{X}_0,x_0)$ we can assume that $\la(X_{r_0-1}) < c < \la(X_{k_0+1})$,
and we have $\la(X_k) \leq \la(Y) < \la(X_{k+1})$ for some $k \in [r_0-1,k_0]$.

Therefore $Y \subset X_{k_0}^+$ for some $X_{k_0}^+ \supset X_{k_0}$ with
either $X_{k_0}^+ = X_{k_0}$ or $\la(X_{k_0}^+) - \la(X_{k_0}) = c -  \la(X_{k_0})$. But then
$$
\la(X_{k_0}^+) - \la(X_{k_0}) \leq |c - \la(X_{k_0})| \leq |c-c_0|+|c_0 - \la(X_{k_0})|\leq |c-c_0|+
\la(X_{k_0}^0 \Delta X_{k_0})
$$
Similarly, $Y \supset X_{r_0}^-$ with $X_{r_0}^- \subset X_{r_0}$ satisfying
$$
\la(X_{r_0}) - \la(X_{r_0}^-) \leqslant |\la( X_{r_0}) - c |\leqslant |c-c_0|+
\la(X_{r_0}^0 \Delta X_{r_0})
$$

Assuming $|c-c_0| < \eps/6$ and $\la(X_k \Delta X_k^0) < \eps/6$ for every $k$, and setting 
$J(\underline{X},c) = Y$, $J(\underline{X}^0,c_0) = Y_0$ 
we get from $X_{r_0}^- \subset Y \subset X_{k_0}^+$ that

$$
\begin{array}{lcl}
\la(Y_0 \Delta Y) &\leq&
 \la(Y_0 \Delta X_{r_0}^-) + \la(Y_0 \Delta X_{k_0}^+) \\
&\leq & \la(X_{r_0}^0 \Delta X_{r_0}^-) + \la(X_{k_0}^0 \Delta X_{k_0}^+) \\
&\leq & (\la(X_{r_0}^0 \Delta X_{r_0}) + \la(X_{r_0} \Delta X_{r_0}^-)) + (\la(X_{k_0}^0 \Delta X_{k_0})+\la(X_{k_0} \Delta X_{k_0}^+))\\
& \leq & 2 |c-c_0| + 2 \la(X_{r_0}^0 \Delta X_{r_0}) + 2 \la(X_{k_0}^0 \Delta X_{k_0}) \\
& \leq & \eps
\end{array}$$
and this proves the continuity
of $J$ at any given $(\underline{X}^0,c_0)$

\end{proof}

\subsection{Proof of item (1) : homotopy equivalence}

We first prove that $L(\KK \times \mathcal{I}) \to L(\KK)$
is an homotopy equivalence, and the analogous statement for pointed complexes.

\begin{proposition} \label{prop:propequivalence}
Let $\KK$ be a simplicial complex over the vertex set $[n]$. The natural projection map $L(\KK \times \mathcal{I}) \to L(\KK)$
is an homotopy equivalence, and $L(\KK) \times \{ 0 \}$ is a deformation retract of $L(\KK \times \mathcal{I})$.
If $\KK$ is considered with $0$ as basepoint, the natural projection map $L(F_{\KK} \wedge F_{\mathcal{I}}) \to L(F_{\KK})=L(\KK)$
is an homotopy equivalence.
\end{proposition}
\begin{proof}
Let us denote $\pi : L(\KK \times \mathcal{I}) \to L(\KK)$ the projection map. It is induced by the map
$[n] \times [1] \to [n]$, $i^{\pm} \mapsto i$. This map admits an inverse on the right $i \mapsto i^-$,
which induces a map $L^1(\Omega,[n]) \to L^1(\Omega,[n]\times [1])$ mapping $L(\KK)$ into $L(\KK \times \mathcal{I})$.
We claim that this map $\sigma : L(\KK) \to  L(\KK \times \mathcal{I})$ is an homotopy inverse
for $\pi$. Since $\pi \circ \sigma = \Id_{L(\KK)}$ this amounts to proving that $\sigma \circ \pi \sim \Id_{L(\KK \times \mathcal{I})}$, and this will prove at the same time that $L(\KK) \times \{ 0 \}$ is a deformation retract of $L(\KK \times \mathcal{I})$.

We now justify the claim by constructing an explicit homotopy $H : L(\KK \times \mathcal{I}) \times I \to L(\KK \times \mathcal{I})$. Let $\underline{A} = (A_j^{\pm})_{j \in [n]}$ and $t \in I$.
Write $t = \frac{1}{n+1}(i + u)$ for some $i \in [n]$ and $u \in [0,1]$,
and set $H(\underline{A},t) = \underline{B}$ with $B_j^{\pm} = A_j^{\pm}$ for $j > i$,
$B_j^{-} = A_j^{+}\cup A_j^-$ and $B_j^+ = \emptyset$ for $j < i$ and finally 
$B_i^- = A_i^- \cup \mathbf{g}(A_i^+,1-u)$, $B_i^+ = A_i^+ \setminus B_i^-$.

It is well-defined, as for $t = i/(n+1) = \frac{1}{n+1}(i + 0) = \frac{1}{n+1}((i-1) + 1)$
we have $B_i^- = A_i^- \cup \mathbf{g}(A_i^+,1-0) = A_i^-$ and
$B_{i-1}^- = A_{i-1}^- \cup A_{i-1}^+ =  A_{i-1}^- \cup \mathbf{g}(A_{i-1}^+,0)$. More concisely,
we have for every $j$ and $t$ the formulas $B_j^- = A_j^- \cup \mathbf{g}(A_j^+,\eta(j+1-(n+1)t))$,
$B_j^+ = A_j^+ \setminus B_j^-$ where $\eta : \R \to \R$ is defined by $\eta(x) = 0$ for $x \leq 0$, $\eta(x) = 1$ for $x \geq 1$ and $\eta(x) = x$ otherwise.

Clearly $H(\underline{A},0) = \underline{A}$, $H(\underline{A},1) = \sigma(\pi(\underline{A}))$
and we have $H(\underline{A},t) \in L(\KK \times \mathcal{I})$ for all $\underline{A},t$. Therefore 
there only remains to prove that $H$ is continuous.
By Lemma \ref{lem:continuiteLOS} this is then a consequence of the continuity of $\mathbf{g}$ and $\eta$,
and of the elementary set-theoretic operations.

We now consider the pointed case, with $0 \in [n]$ chosen for basepoint of $\KK$. By construction the
smash product
$F_{\KK} \wedge F_{\mathcal{I}}$ is a push-out hence a colimit. 
 Since $L$ is a left adjoint it commutes with colimits, and from this we
get that $L(F_{\KK} \wedge F_{\mathcal{I}^+})$ is the push-out of
$$
\xymatrix{
L(\KK \vee \mathcal{I}^+) \ar[r]\ar[d] & L( \KK \times \mathcal{I}^+)\ar[d] \\
\mbox{*} \ar[r] & \bullet
}
$$
that is the quotient of  $L( \KK \times \mathcal{I}^+)$ by the subspace 
$L(\KK \vee \mathcal{I}^+)$. Now notice that the projection map $L(\KK \times \mathcal{I}) \to L(\KK)$
factorizes through the maps $L(\KK \times \mathcal{I}) \to L(F_{\KK} \wedge F_{\mathcal{I}^+})$
and $L(F_{\KK} \wedge F_{\mathcal{I}^+}) \to L(\KK)$. We want to prove that the latter is an homotopy
equivalence. Since we just proved that the composite $L(\KK \times \mathcal{I}) \to L(\KK)$ is an homotopy
equivalence, it is equivalent to prove that the former map $L(\KK \times \mathcal{I}) \to L(F_{\KK} \wedge F_{\mathcal{I}^+})$ is an homotopy equivalence. Since $\KK \times \mathcal{I}^+
= (\mathcal{K} \times \mathcal{I}) \cup (\KK \times \{ * \})$ and
$\KK \times \{ * \} \subset \KK \vee \mathcal{I}^+ $
we get that this map is the same as the quotient map $\pi_0 : L(\KK \times \mathcal{I}) \to L(\KK \times \mathcal{I})/L_0$ with $L_0 = L^1(\Omega, \{ 0^-,0^+ \})$.

We consider the map $L(\KK \times \mathcal{I}) \to L(\KK \times \mathcal{I})$ mapping $\underline{A}^{\pm}$
to $\underline{B}^{\pm}$ with $B_i^{\pm} = A_i^{\pm}$ for $i > 0$,
$$
B_0^- 
= A_0^- \cup \mathbf{g}\left(A_0^+,\sum_{\stackrel{i \neq 0}{\eps = \pm }}\la(A_i^{\pm})\right)
= A_0^- \cup \mathbf{g}\left(A_0^+,1 - \la(A_0^- \cup A_0^+) \right)
$$
and $B_0^+ = A_0^+ \setminus B_0^-$. If $\underline{A}^{\pm} \in L_0$, that is if $i > 0 \Rightarrow A_i^{\pm} = \emptyset$, then $B_0^- = A_0^- \cup A_0^+ = \Omega$, $B_0^+ = \emptyset$. It follows that the map
factorizes through $L(\KK \times \mathcal{I})/L_0$ and induces a map $\sigma_0 : L(\KK \times \mathcal{I})/L_0 \to
L(\KK \times \mathcal{I})$.

Consider the map $H : L(\KK \times \mathcal{I}) \times I \to L(\KK \times \mathcal{I})$ mapping $(\underline{A}^{\pm},t)$ to $\underline{B}^{\pm}$ with $B_i^{\pm} = A_i^{\pm}$ for $i > 0$
and 
$$B_0^- = A_0^- \cup \mathbf{g}\left(A_0^+,t(1 - \la(A_0^- \cup A_0^+))+(1-t) \right)$$
It provides an homotopy between $\sigma_0 \circ \pi_0$ and the identity map. Now notice that
the image of $L_0 \times I$ is equal to $L_0$, so that it induces a map 
$L(\KK \times \mathcal{I})/L_0 \times I \to L(\KK \times \mathcal{I})/L_0$. This map defines an homotopy
between $\pi_0 \circ \sigma_0$ and the identity map. This proves that $\pi_0$ is an homotopy equivalence.

\end{proof}

\subsection{Proof of item (2) : cofibration property}

We assume that $\mathcal{K}$ is a subcomplex of $\PF^*([n])$, so that
$L(\KK \times \mathcal{I})$ (resp. $L(\KK \wedge \mathcal{I})$) is a closed subset of $L(\PF^*([n]) \times \mathcal{I}) \subset L^1(\Omega,[n] \times [1])$ (resp. $L(\PF^*([n]) \wedge \mathcal{I}) \subset L^1(\Omega,[n] \times [1])$). Precisely, it is made of the $f \in L(\PF^*([n]) \times \mathcal{I})$ (resp. $f \in L(\PF^*([n]) \wedge \mathcal{I})$) such that $\pi(f) \in L(\KK)$,
for $\pi : L(\KK \times \mathcal{I}) \to L(\KK)$ (resp. $\pi : L(\KK \wedge \mathcal{I}) \to L(\KK)$) the natural projection map, that is such that $\pi(f)(\Omega) \in \KK$. 
\begin{figure}
\begin{center}
\begin{tikzpicture}[scale=3]
\begin{scope}
\draw[dotted] (0,1) rectangle (1,0);
\draw[red,thick] (0,1) -- (0,0) -- (1,0) -- (1,1);
\fill[black] (0,0) circle (.03);
\fill[black] (1,0) circle (.03);
\draw (1,-.15) node {$(1,0)$};
\draw (0,-.15) node {$(0,0)$};
\draw (.2,.9) node {$\bullet$};
\draw (.2+.2,.9) node {$(y,u)$};
\draw (.5,2) node {$\times$};
\draw (.5,2+.2) node {$(\frac{1}{2},2)$};
\draw[dashed] (.5,2) -- (0,0.05/0.3);
\draw (0,0.05/0.3) node {$\bullet$};
\draw (0.22,0.05/0.3) node {$(y',u')$};
\end{scope}
\begin{scope}[shift={(1.5,0)}]
\fill[green] (0,0) -- (0,1) -- (.25,1) -- cycle;
\fill[yellow] (0,0) -- (.25,1) -- (.75,1) -- (1,0) -- cycle;
\fill[blue] (1,0) -- (.75,1) -- (1,1) -- cycle;
\draw[dotted] (0,1) rectangle (1,0);
\draw[red,thick] (0,1) -- (0,0) -- (1,0) -- (1,1);
\fill[black] (0,0) circle (.03);
\fill[black] (1,0) circle (.03);
\draw (1,-.15) node {$(1,0)$};
\draw (0,-.15) node {$(0,0)$};
\draw (.5,2) node {$\times$};
\draw (.5,2+.2) node {$(\frac{1}{2},2)$};
\draw[dashed] (.5,2) -- (0,0);
\draw[dashed] (.5,2) -- (1,0);
\fill[green] (1.3,0) rectangle (1.5,.2);
\draw (2.32,.1) node {$(y,u) \mapsto (0,\frac{u-4y}{1-2y}) \mbox{ for } u \geq 4y$};
\fill[yellow] (1.3,.3) rectangle (1.5,.5);
\draw (2.05,.4) node {$(y,u) \mapsto (\frac{1}{2}\frac{u-4y}{u-2},0)$};
\fill[blue] (1.3,.6) rectangle (1.5,.8);
\draw (2.53,.7) node {$(y,u) \mapsto (1,\frac{4-4y-u}{1-2y})  \mbox{ for } u \geq 4(1-y)$};
\end{scope}
\end{tikzpicture}
\end{center}

\caption{The geometric retract $q : [0,1]^2 \to \{ (y,u) ; u = 0 \mbox{ or } y \in \{ 0, 1\}\}$}
\label{fig:qretract}
\end{figure}
We now prove the following.

\begin{proposition}
\label{prop:propcofib}
 Let $\KK$ be a simplicial complex with vertex set $[n]$. Then the inclusion
map $L(\KK) \times \{ 0, 1 \} \to L(\KK \times \mathcal{I})$ is a closed cofibration.
For $\KK$ considered as a pointed simplicial complex, the inclusion map $L(\KK \vee \KK) \to L(F_{\KK} \wedge F_{\mathcal{I}^+})$
is also a closed cofibration.
\end{proposition}

\begin{proof}
The space $L(\KK) \times \{0,1 \}$ is equal to $L(\KK \times \{ 0, 1 \} )$,
where we denote $\KK\times \{ 0, 1\} = \KK \sqcup \KK$
the disjoint union of two copies of $\KK$, with vertex sets $\{ x^+, x \in [n]\}$
and $\{ x^-, x \in [n]\}$.
Since $L(\KK \times \{ 0, 1 \} )$ is a closed subset of $L(\KK \times \mathcal{I})$, it is enough to
construct a retract of the cylinder $L(\KK \times \mathcal{I}) \times I$ onto $(L(\KK) \times \{ 0, 1 \} \times I) \cup L(\KK \times \mathcal{I}) \times \{ 0 \}$. We do this now.

We endow the space $\ultriangle_n = \{ \underline{X} = (X_1,\dots,X_n) ; \emptyset \subset X_1 \subset
X_2 \subset \dots \subset X_n \subset \Omega \}$ with the
metric $d(\underline{X}^1,\underline{X}^2) = \max_i \la(X_i^1 \Delta X_i^2)$,
and $L(\PF^*([n]) \times \mathcal{I})$ with the $L^1$ metric.

Let $\mathcal{F}_r = \PF^*(\{ 0^-,1^-,\dots,r^-,r^+,(r+1)^+,\dots, n^+ \}$ for $0 \leq r \leq n$. These are
the maximal simplices of $\PF^*([n]) \times \mathcal{I}$, and $L(\PF^*([n]) \times \mathcal{I}) = \bigcup_r L(\mathcal{F}_r)$.

We define a map $D : L(\PF^*([n]) \times \mathcal{I}) \to \ultriangle_{n+1}$ as
follows. To $f \in L(\mathcal{F}_r)$ we associate $\underline{X} \in \ultriangle_n$ such
that $X_i = f^{-1}([0^-,(i-1)^-])$ if $i \leq r+1$, and $X_i = f^{-1}([0^-,(i-2)^+])$ if $i > r+1$. We check that
these definitions agree on $L(\mathcal{F}_r) \cap L(\mathcal{F}_s)$ for $r < s$ as follows. If $i \leq r+1$ or $i > s+1$
the two possible definitions of $X_i$ clearly agree, so we can assume $r+1 < i \leq s+1$. According to
the definition on $L(\mathcal{F}_r)$ we have $X_i = f^{-1}([0^-,(i-2)^+])$,
whereas according to 
the definition on $L(\mathcal{F}_s)$ we have $X_i = f^{-1}([0^-,(i-1)^-])$. But since $f \in L(\PF^*([n]) \times \mathcal{I})$ we have $\# f(\Omega) \cap \{ (i-2)^+, (i-1)^- \} \leq 1$
hence $f^{-1}([0^-,(i-2)^+]) = f^{-1}([0^-,(i-1)^-])$ and this proves that these maps can be glued together.
Finally it is easily checked that the map $D$ is 2-Lipschitz on each $L(\mathcal{F}_r)$, as
we have, for $\underline{X} = D(f)$ and $\underline{X}^0 = D(f_0)$,
$$
\la(X_i \Delta X_i^0) = \la(X_i \setminus X_i^0) + \la(X_i^0 \setminus X_i)
= \int_{X_i} d(f(t),f_0(t))\dd t + \int_{X_i^0} d(f(t),f_0(t))\dd t
\leqslant 2 d(f,f_0)
$$
 and this proves that
it is continuous on its whole domain. Moreover it is easily checked that every restriction 
$D_{|L(\mathcal{F}_r)} : L(\mathcal{F}_r) \to \ultriangle_{n+1}$ is a bijection.

Let $\tau_{+} : [n] \times [1] \to [n]$ defined by $i^{\pm} \mapsto i^+$
and similarly $\tau_- : i^{\pm} \mapsto i^-$.
We then define two maps $G_{\pm} : L(\PF^*([n]) \times \mathcal{I}) \times I \to L(\PF^*([n]) \times \mathcal{I})$ as follows.
Let $(f,c) \in L(\PF^*([n]) \times \mathcal{I}) \times I$. We use the map $J$ of Lemma \ref{lem:interpolationJ}.

For $x\in J(D(f),u)$ we
set $G_{+}(f,u)(x) = f(x)$, and for $x \not\in J(D(f),u)$, we set $G_{+}(f,u)(x) = \tau_{+}(f(x))$.
For $x\not\in J(D(f),u)$ we
set $G_{-}(f,u)(x) = f(x)$, and for $x \in J(D(f),u)$, we set $G_{-}(f,u)(x) = \tau_{-}(f(x))$.

Let $(f_0,c_0) \in L(\PF^*([n]) \times \mathcal{I}) \times I$. We prove that $G_{\pm}$
is continuous at $(f_0,c_0)$. Let us consider $(f,c) \in L(\PF^*([n]) \times \mathcal{I}) \times I$
and set $J_0 = J(D(f_0),c_0)$, $J=J(D(f),c)$. Let $\eps > 0$, and assume $d(f,f_0) < \eps/3$. By continuity of $D$ and $J$ we know that,
for $(f,c)$ close enough to $(f_0,c_0)$, we have $\la(J \Delta J_0) \leq \eps/3$. Then, by definition
$$
d(G_{+}(f,c),G_{+}(f_0,c_0))\leqslant \la(J \Delta J_0)+\int_{J \cap J_0}d(f(t),f_0(t))\dd t +\int_{\Omega\setminus(J \cup J_0)}d(\tau_{+}(f(t)),\tau_{+}(f_0(t)))\dd t 
$$
Since $d(\tau_{+}(x),\tau_{+}(y)) \leq d(x,y)$ for every $x,y$, this implies
$d(G_{+}(f,c),G_{+}(f_0,c_0)) \leq \eps/3 + 2 d(f,f_0) \leq \eps$
and this proves the continuity of $G_{+}$ at $(f_0,c_0)$. The proof of continuity for $G_-$ is similar and left to the reader.

We extend the maps $\tau_{\pm}$ to $L^1(\Omega,[n] \times [1])$
by setting $\tau_{\pm} (f) = x \mapsto \tau_{\pm}(f(x))$. It is clear from the definitions
that 
\begin{itemize}
\item
$G_{+}(f,1) = f$ and $G_{+}(f,0) = \tau_{+}(f)$.
\item
$G_{-}(f,1) = \tau_-(f)$ and $G_{-}(f,0) = f$.
  \end{itemize}

We now construct a continuous map $L(\PF^*([n])\times \mathcal{I}) \times I \to L(\PF^*([n])\times \mathcal{I}) \times I$ as follows. To $(f,u) \in L(\PF^*([n])\times \mathcal{I}) \times I $
we associate $(p(f),u) = (\underline{x},y,u) \in (\lltriangle_n \times \lltriangle_1) \times \Delta_1$,
and then $q(y,u) = (y',u')$, where $q$ is the projection map of Figure \ref{fig:qretract}. Notice that, if $y = 1/2$, then $y' = 1/2$ and there exists
a component $X_k$ of $D(f)$ of measure $1/2$, which implies $J(D(f),1/2) = X_k$ and finally $G_{\pm}(f,1/2) = f$. 
Therefore, setting $H(f,u) = (G_+(f,y'),u')$ for $y \leq 1/2$
and   $H(f,u) = (G_-(f,y'),u')$ for $y > 1/2$, we get a well-defined continuous map
with the properties that :
\begin{itemize}
\item either $u' = 0$, in which case $H(f,u) \in  L(\PF^*([n])\times \mathcal{I}) \times \{ 0 \}$,
\item or $u' > 0$ in which case either $y' = 0$ and $H(f,u) = (G_+(f,0),u')= (\tau_+(f),u') \in L(\PF^*([n]\times \{ 1 \})$,
or $y' = 1$ and  $H(f,u) = (G_-(f,1),u') = (\tau_-(f),u')\in L(\PF^*([n]\times \{ 0 \})$.
\item if $u = 0$, then $u' = 0$ and $y' = y$, in which case $J(D(f),y) = X_k$ for some component $X_k$
of measure $y' = y$, and this yields $H(f,u) = f$.
\item if $y = 0$, then $f = \tau_+(f)$, $y' = y$ and $H(f,u) = \tau_+(f) = f$
\item if $y = 1$, then $f = \tau_-(f)$, $y' = y$ and $H(f,u) = \tau_-(f) = f$
\end{itemize}
This proves that $H$ provides a suitable retract, for $\KK = \PF^*([n])$. In the general case,
starting from the map $H : L(\PF^*([n])\times \mathcal{I}) \times I \to L(\PF^*([n])\times \mathcal{I}) \times I$
we constructed, we consider its restriction $H_{\KK}$ to $L(\KK\times \mathcal{I}) \times I$.
Starting from $(f,u) \in L(\KK\times \mathcal{I}) \times I$, that is
with $\pi(f)(\Omega) \in \KK$, we need to check that $\pi(G_{\pm}(f,y'))(\Omega) \in \KK$
so that $H_{\KK}$ provides a suitable retract. Since by construction $\pi \circ G_{\pm}(f,y') = f$, this is obvious.

We now consider the pointed case. From the decomposition $\KK \times \mathcal{I}^+ = (\KK \times
\mathcal{I}) \cup (\KK \times \{ * \})$ observed in the previous section, we can again identify
$L(F_{\KK} \wedge F_{\mathcal{I}^+})$ with the quotient space $L(\KK \times \mathcal{I})/L_0$ where $L_0 = L^1(\Omega,\{ 0^-,0^+\})= L(\PF^*(\{0^-,0^+ \}))$.

Then $L(\KK \vee \KK)$ can be identified with a closed subset of 
$L(\KK \times \mathcal{I})/L_0$
via $0^{\pm} \mapsto 0^-$, $i^{\pm} \mapsto i^{\pm}$ for $i > 0$, and it is 
precisely the image of $L(\KK \times \{ 0, 1\})$ under the composite of the natural maps
$L(\KK \times \{ 0, 1 \}) \to L(\KK \times \mathcal{I})$ and $L(\KK \times \mathcal{I}) \to L(\mathcal{K}\times
\mathcal{I})/L_0$.
 Now notice that, for every $f,u,x$ we have
$G_{\pm}(f,u)(x) \in \{ f(x),\tau_+(f(x)),\tau_-(f(x)) \}$ and this proves that $G_{\pm}(L_0 \times I) \subset L_0 $.
This implies that $H_{\KK}$ induces a map $\bar{H}_{\KK} : (L(\KK \times \mathcal{I})/L_0)\times I
\to (L(\KK \times \mathcal{I})/L_0)\times I$. It is easily checked that this
map is a retract onto $(L(\KK \vee \KK) \times I) \cup (L(\KK \times \mathcal{I})/L_0) \times \{ 0 \}$, and this proves the claim in the pointed case.

\end{proof}

\section{Quillen equivalences between $\sSet$ and $\Top$}
\label{sect:quillen}

In this section we want to prove that the functors $L$ and $|\bullet|$ from $\sSet$ to $\Top$
induce the same functor $\HosSet \to \HoTop$, where $\HosSet$ and $\HoTop$ denotes the homotopy
categories obtained from $\sSet$ and $\Top$ by formally inverting the weak homotopy equivalences.

In order to this, we use the classical (closed) model category structure on $\sSet$, and endow $\Top$
with a model category structure more flexible than Quillen's classical one, which nevertheless defines the same homotopy category. 
Such a model category has been introduced by M. Cole in \cite{COLE}.
Actually, Cole presents his construction as an intermediate between Quillen's and Str\o m's model category structures on $\Top$. Our point of view here is more that is provides a `flexibilization' of Quillen's model category structure on $\Top$, relevant
to the same homotopy category. Notice for instance that the subcategory
of cofibrant objects is Milnor's category $\mathcal{W}$ of spaces
which are \emph{homotopically equivalent} to CW-complexes, instead
of the category of genuine CW-complexes.

We recall the main characteristics of this model category structure below. With $\sSet$ and $\Top$ endowed with these structures of model categories, the main theorem
of this section is the following one.

\begin{theorem} \label{theo:QuillenEq} The functors $L : \sSet \to \Top$ and $\Sing_{RV} : \Top \to \sSet$ together provide a Quillen
equivalence. In particular they provide an equivalence of categories between their homotopy categories.
\end{theorem}

We shall make clear in Section \ref{sect:colemodel} that the classical functors
$|.|$ and $\Sing$ also provide a Quillen equivalence for the \emph{same}
model category structures.
The situation is thus summarized by the following diagram, where $p : L \to |.|$ and
$R : \Sing \to \Sing_{RV}$ are the obvious natural transformations induced by
the probability-law maps, so that the probability-law map somewhat
provides a \emph{natural transformation between the two Quillen equivalences}.

\begin{center}

\begin{tikzpicture}
\node (a) at (0,0) {$\sSet$};
\node (b) at (2.5,0) {$\Top$};
\draw[->] (a) to [bend left] node [scale=.7] (f) [below] {$|\bullet |$} (b);
\draw[->] (a) to [bend left=80] node [scale=.7] (F) [above] {$L$} (b);
\draw[-{Implies}, double distance =1.5 pt, shorten >= 2 pt,shorten <= 2 pt] (F) to node [scale=.7] [right] {$p$} (f);
\draw[->] (b) to [bend left] node [scale=.7] (g) [above] {$\Sing$} (a);
\draw[->] (b) to [bend left=80] node [scale=.7] (G) [below] {$\Sing_{RV}$} (a);
\draw[-{Implies}, double distance =1.5 pt, shorten >= 2 pt,shorten <= 2 pt] (g) to node [scale=.7] [right] {$R$} (G);

\end{tikzpicture}
\end{center}
\subsection{Cole's model category structure on $\Top$}
\label{sect:colemodel}

From \cite{COLE} Theorem 2.1 one gets that one can mix Quillen's classical model category structure on $\Top$,
which has for fibrations and weak equivalences the Serre fibrations and weak homotopy equivalences, respectively,
with Str\o m's structure, which has for fibrations and weak equivalences the Hurewicz fibrations and (strong)
homotopy equivalences, respectively. The resulting model category, which we call Cole's model category, has for fibrations the Hurewicz fibrations and for
weak equivalences the weak homotopy equivalences. In particular, it has the same
homotopy category as Quillen's original one, and moreover a Quillen adjunction (resp. equivalence) between $\sSet$ and $\Top$ for Quillen's original model category structure is a Quillen adjunction (resp. equivalence) between $\sSet$ and $\Top$ for Cole's
model category structure : indeed, every (trivial) cofibration for
Quillen's model category structure is in particular a (trivial) cofibration
for Cole's model category -- where we use the terminology that a (co)fibration is said to be a trivial
one if it is in addition a weak homotopy equivalence. 

More precisely, recall that the cofibrations of Quillen's classical model category structure on $\Top$
are (retracts of) CW-attachments (also called relative CW complexes). We call them Quillen-cofibrations.
The cofibrations in Cole's model category structure are the closed cofibrations (in the classical sense)
$f$ which can be written as $\xi \circ f'$ where $f'$ is a Quillen-cofibration and $\xi$ is an homotopy equivalence
(see \cite{COLE} Proposition 3.6 and Example 3.8). We call them Cole cofibrations.

The first application of this is for the classical adjunction $| \bullet | : \sSet \leftrightarrow \Top : \Sing$, which provides
a Quillen equivalence for Quillen's model category structure on $\Top$,
and therefore \emph{also} for Cole's model category structure.

\subsection{Proof of Theorem \ref{theo:QuillenEq}}

We shall prove the following strengthening of Proposition \ref{prop:cofibnabla}.

\begin{proposition} \label{prop:cofibnablacole} For every $n \geq 0$ the inclusion map $\partial \nabla_n \to  \nabla_n$ is a Cole cofibration.
\end{proposition}

Then recall that the $r$-th hook of the $n$-dimensional simplex
has for geometric realization
$$\Lambda_n^r = \{ (a_0,\dots,a_n) \in \Delta_n \ | \ \exists i\neq r \  \ a_i = 0 \} \subset \partial
\Delta_n \subset \Delta_n.
$$
Its image under the functor $L$ is
$$
V_n^r = \{ (A_0,\dots,A_n) \in \nabla_n \ | \  \exists i\neq r\  \ A_i = \emptyset \} \subset \partial
\nabla_n \subset \nabla_n.
$$
In this subsection we are going to prove the following. Recall
that, in any given model category, a (co)fibration is said to be a trivial
one if it is in addition a weak homotopy equivalence.

\begin{proposition} \label{prop:cofibrationfaiblecornet}
For $n >0$ the inclusion $V_n^r \subset \nabla_n$ is a trivial cofibration for Cole's model category
structure on $\Top$.
\end{proposition}

From these two propositions and our previous results we finally can prove Theorem \ref{theo:QuillenEq}.

\begin{proof}

Recall (see e.g. \cite{HOVEY} Theorem 3.6.5) that the category of simplicial sets is a finitely generated closed
model category with generating cofibrations the inclusions of simplicial complexes $\partial \PF^*([n]) \subset
\PF^*([n])$ for $n \geq 0$ and generating trivial cofibrations the inclusion of
simplicial complexes $\Lambda_n^r \subset \PF^*([n])$ for $n > 0$ et $r \in [n]$.
By the above propositions we know that their image under $L$ are cofibrations and trivial cofibrations, respectively.
Therefore (\cite{HOVEY} Lemma 2.1.20) $L : \sSet \to \Top$ and $\Sing_{RV} : \Top \to \sSet$ together provide a Quillen
adjunction. In order to prove that they provide a Quillen equivalence, we need to prove
that $L$
maps cofibrations to cofibrations and
trivial cofibrations to trivial cofibrations
for Cole's model category structure
hence provides a Quillen adjunction. 
Since every object of $\sSet$ is cofibrant and every object of $\Top$ is fibrant for Cole's model category (the constant maps are Hurewicz fibrations), it remains to prove (\cite{HOVEY} Corollary 1.3.16)
that $L$ reflects weak equivalences and that $L(\Sing_{RV} X) \to X$ is always a weak equivalence.

Let $f : F \to G$ be such a weak equivalence in $\sSet$. By definition this is equivalent to asking for
the induced map $|f| : |F| \to |G|$ to be one. Applying the natural transformation $p : L \leadsto |\bullet|$,
we get the commutative square
$$
\xymatrix{
|F| \ar[r]^{|f|} & |G| \\
L(F) \ar[r]^{L(f)}\ar[u]_{p_F} & L(G)\ar[u]_{p_G} \\
}
$$
and we know by Theorem \ref{theo:LwFw} that $p_F$ and $p_G$ are homotopy equivalences. Therefore
$L(f)$ is a weak equivalence iff $|f|$ is a weak equivalence.

We finally need to prove that $L(Sing_{RV} X) \to X$ is a weak equivalence for every $X \in \Top$. We consider the following diagram, where the horizontal maps are induced from the weak equivalence $\Sing X \to \Sing_{RV}X$ of Theorem \ref{theo:eqSingSingRV}, the vertical maps by the natural transformation $p$, and the diagonal ones are the obvious co-unit maps.

$$
\xymatrix{
L(\Sing X) \ar[r] \ar[d] & L(\Sing_{RV} X) \ar[d] \ar[ddr] \\
|\Sing X| \ar[r] \ar[drr] & |\Sing_{RV} X | \\
& & X
}
$$
Since $|\Sing X| \to X$ is a classical weak equivalence and since $L(\Sing X) \to |\Sing X|$ is also
one by Theorem \ref{theo:LwFw}, we get that $L(\Sing X) \to X$ is one. Therefore the composite of the
maps $L(\Sing X) \to L(\Sing_{RV} X)$ and $L(\Sing_{RV} X) \to X$ is a weak-equivalence. Since
$\Sing X \to \Sing_{RV} X$ is a weak equivalence by Theorem \ref{theo:eqSingSingRV} so is $L(\Sing X) \to L(\Sing_{RV} X)$ as we just
showed, and this has for consequence that $L(\Sing_{RV} X) \to X$ is a weak-equivalence, which concludes the proof.

\end{proof}

\subsection{Proof of Proposition \ref{prop:cofibnablacole}}

Since $\partial \nabla_0 = \emptyset$ the statement is trivial for $n=0$ and
we can assume $n\geq 1$.
We already know from Proposition \ref{prop:cofibnabla} that the inclusion map $\partial \nabla_n \to \nabla_n$
is a cofibration. We prove here that it can be written as the composition of a cellular attachement with
an homotopy equivalence.

We consider the 
section $\sigma_n : \Delta_n \to \nabla_n$ of the probability-law map $p_n : \nabla_n \to \Delta_n$ used in Section \ref{sect:presimpl}. We recall
from there that $\sigma_n \circ p_n$ is homotopic
to the identity map, by an homotopy $H_n$ which commutes to the face maps for various $n$'s (see Proposition \ref{prop:commuteface}). Clearly, $\sigma_n, p_n, H_n$ preserve the boundaries of $\nabla_n$ and $\Delta_n$.
In particular we get an attaching map $\partial \sigma_n : \partial \Delta_n \to \partial \nabla_n$, enabling us to
build a relative cellular complex $\partial \nabla_n \subset (\partial \nabla_n)\cup_{\partial \sigma_n} \Delta_n$.
Moreover, our map $\partial \nabla_n \to \nabla_n$ is the composition of this Quillen cofibration with
the map $ f' : (\partial \nabla_n)\cup_{\partial \sigma_n} \Delta_n \to \nabla_n$ obtained by gluing together
the inclusion map $\partial \nabla_n \to \nabla_n$ with the map $\sigma_n : \Delta_n \to \nabla_n$. It remains
to prove that $f'$ is an homotopy equivalence.

Consider the map $p_n \circ f' : (\partial \nabla_n)\cup_{\partial \sigma_n} \Delta_n \to \Delta_n$. It
is equal to the identity map on $\Delta_n$, and to $p_n$ on $\partial \nabla_n$. It admis a section given
by $\sigma_n$ on $\partial \Delta_n$ and to the identity on $\Delta_n$. Moreover, since the homotopy $H_n$ preserves
the boundary, we get that $p_n \circ f'$ is an homotopy equivalence. Since $p_n$ is also an homotopy equivalence,
from the 2-out-of-3 property of homotopy equivalences we get that $f'$ is an homotopy equivalence, which concludes the proof.

\subsection{Proof of Proposition \ref{prop:cofibrationfaiblecornet}}

We prove it in several steps.

\begin{proposition} \label{prop:VnrNablanstrongdefretract}
For every $n \geq 1$ the inclusion $V_n^r \subset \nabla_n$ is a strong deformation retract. In particular, it is an
homotopy equivalence.
\end{proposition}
\begin{proof}
We define $H : \nabla_n \times I \to \nabla_n$ by $H(\underline{A},t) = \underline{B}$
with 
$$
B_i = \mathbf{g}\left(A_i,t \frac{\min_{j \neq i,r} \la(A_j)}{\la(A_i)}\right)
$$
for $i \neq r$ with the convention $0/0 = 1$, and $B_r = \Omega \setminus \bigcup_{i \neq r} B_i$.
From the continuity of $\mathbf{g}$ and of the elementary set-theoretic operations
one gets immediately that $H$ is continuous.
Since $B_i \subset A_i$ for $i \neq r$,
we have $H(V_n^r\times I) \subset V_n^r$.
Moreover, for $t = 0$ we have $B_i = A_i$ for all $i$ that is $H(\underline{A},0) = \underline{A}$ for every $\underline{A} \in \nabla_n$.

Consider $\underline{A} \in V_n^r$. By definition there exists $i_0 \neq r$ with $A_{i_0} = \emptyset$.
From $A_{i_0} = \emptyset$ one gets $B_{i_0} = \emptyset$, and also
$\min_{j \neq i,r} \la(A_j) = 0$ hence $B_j = \mathbf{g}(A_j,0) = A_j$
for every $j \neq r$. It follows that $B_r = A_r$ hence $H(\underline{A},t) = \underline{A}$ for every $\underline{A} \in V_n^r$ and $t \in I$.

Now assume $t = 1$. 
Notice that there exists $i \neq r$ such that $\la(A_i) = 
\min_{j \neq i,r} \la(A_j)$.
 Then, for $t = 1$,
$B_i = \mathbf{g}(A_i,1) = \emptyset$, and this proves $H(\underline{A},1)  \in V_n^r$.

Therefore $H$ indeed provides a strong deformation retract.

\end{proof}

\begin{proposition}
For $n \geq 1$ the inclusion $V_n^r \subset \nabla_n$ is a closed cofibration.
\end{proposition}
\begin{proof}
We already know that $V_n^r$ is a closed subset of $\nabla_n$. We need to exhibit
a retract of $\nabla_n \times I$ onto $(\nabla_n \times \{ 0 \}) \cup (V_n^r \times I)$.
We define $H : \nabla_n \times I \to (\nabla_n \times \{ 0 \}) \cup (V_n^r \times I)$ as follows.
First of all, define a continuous map $m : \nabla_n \to [0,1]$ by $m(\underline{A}) = \min_{i \neq r} \la(A_i)$,
and then define $H(\underline{A},u) = ((H(\underline{A},u)_i)_{i \in [n]},\max\left(u-m(\underline{A}),0\right))$ by setting
$$H(\underline{A},u)_i = 
\mathbf{g}\left(A_i, \frac{m(\underline{A}) + \min( u-m(\underline{A}),0)}{\la(A_i)}\right)
$$
for $i \neq r$, $H(\underline{A},u)_r = \Omega \setminus \bigcup_{i \neq r} H(\underline{A},u)_i$.
This formula is made so that $H(\underline{A},u)_i$ for $i \neq r$ has measure $\la(A_i) - m(\underline{A})$ for $u \geq m(\underline{A})$ and $\la(A_i) -u$ for $u \leq m(\underline{A})$.
In particular, $H(\underline{A},u) \in V_n^r \times I$ for $u \geq m(\underline{A})$
and $H(\underline{A},u) \in \nabla_n \times \{ 0 \}$ for $u \leq m(\underline{A})$.
Moreover, for $\underline{A} \in V_n^r$ we have $m(\underline{A})=0$ hence $H(\underline{A},u)_i
= \mathbf{g}(A_i,0) = A_i$ for every $i \neq r$, whence $H(\underline{A},u) = (\underline{A},u)$.
Similarly, for $u = 0$, we have $\max\left(u-m(\underline{A}),0\right) = 0$
and $H(\underline{A},u)_i =  \mathbf{g}(A_i,0) = A_i$ for every $i \neq r$, $H(\underline{A},0) = (\underline{A},0)$.
Since $H$ is continuous as a composite of continuous maps, this provides a suitable retract.
\end{proof}

In order to prove that these inclusions $V_n^r \to \nabla_n$ are Cole cofibrations, it remains to
prove that they can be obtained as the composition of a Quillen cofibration with an homotopy
equivalence. We consider the restriction $ \sigma_{n-1}^{(r)} : \Delta_{n-1} \to \nabla_n$ of $\sigma_n : \Delta_n \to \nabla_n$ to the $r$-th face
of its boundary, and then its restriction $\partial \sigma_{n-1}^{(r)}$ to the boundary of the $r$-th face. This provides an attachment map $\partial \sigma_{n-1}^{(r)} : \partial \Delta_{n-1} \to V_n^r$,
mapping $\underline{a} \in \partial \Delta_{n-1}$ to $\sigma_{n-1}^{(r)}(\underline{a}) = \sigma_n(\Delta(D_r^c)(\underline{a}))$.

 Let $Y = V_n^r \cup_{\partial \sigma_{n-1}^{(r)}} \Delta_{n-1}$. Since $V_n^r$ and the $r$-th face $\nabla_{n-1}$
 of $\partial \nabla_n$ are closed inside $\partial \nabla_n$, the set-theoretic decomposition
$\partial \nabla_n = V_n^r \cup \nabla_{n-1}$ is a topological union. Therefore there is a  
map $f' : Y \to \partial \nabla_n$
 obtained by mapping an element of $V_n^r$ to itself, and $\underline{a} \in \Delta_{n-1}$ to $\sigma_n(\Delta(D_r^c)(\underline{a}))$.

We then consider a map $\partial'\sigma_n : \partial \Delta_n \to Y$ given by mapping the $r$-th face identically to the copy of $\Delta_{n-1}$ just added, and by applying $\sigma_n$ to the other faces.
In formulae, $\underline{a}=(a_0,\dots,a_n) \in \partial \Delta_n$ is mapped to $(a_0,\dots,a_{r-1},a_{r+1},\dots,a_n) \in \Delta_{n-1}$ if $a_r = 0$, and to $\sigma_n(\underline{A}) \in V_n^r$ otherwise. We let $Z = Y  \cup_{\partial'\sigma_n} \Delta_n$. Again because $\partial \nabla_n$ is closed inside $\nabla_n$, we can
write $\nabla_n$ as a topological union $(V_n^r \cup \nabla_{n-1})\cup \nabla_n$ and
there is a map $f'' :  Z \into \nabla_n$
obtained by gluing $f' : Y \to \nabla_n$ with $\sigma_n : \Delta_n \to \nabla_n$.
Clearly, the restriction of $f''$ to $V_n^r$ is the inclusion map $V_n^r \to \nabla_n$.

The inclusion map $V_n^r \into Z$ being a relative CW-complex,
proving that the inclusion $V_n^r \to \nabla_n$ is a Cole cofibration amounts to proving that
$f''$ is an homotopy equivalence.

We post-compose it with $p_n : \nabla_n \to \Delta_n$ and get a map
$$
Z = (V_n^r \cup \Delta_{n-1}) \cup \Delta_n \stackrel{p'}{\to} (\Lambda_n^r \cup \Delta_{n-1})\cup \Delta_n = \Delta_n
$$
equal to $p_n$ on $V_n^r$ and to the identity map on the rest This map $p'$ admits a section $\sigma'$
defined by $\sigma_n$ on $\Lambda_n^r$ and to the identity map on the rest. In order to prove
that $\sigma' \circ p'$ is homotopic to the identity map, we build $H_Z : I \times Z \to Z$ via the decomposition
$$
I \times Z = (I \times V_n^r) \cup (I \times \Delta_{n-1} ) \cup (I \times \Delta_n )
$$
gluing the map $H_n $ of Proposition \ref{prop:commuteface} on $I \times V_n^r$ with the identity map on the two other parts.
This is possible because $\forall \underline{a} \in \Delta_n \  H_n(t,\sigma_n(\underline{a})) = \sigma_n(\underline{a})$ by Proposition \ref{prop:commuteface}. This provides an homotopy proving that $p' = p_n \circ f''$ is an homotopy
equivalence. Since $p_n$ is an homotopic equivalence so is $f''$ by the 2-out-of-3 principle and this proves the claim.

\end{document}